\documentclass{birkjour}
 \newtheorem{thm}{Theorem}[section]
 
 \newtheorem{lem}[thm]{Lemma}
 \newtheorem{prop}[thm]{Proposition}
 \newtheorem{ass}[thm]{Assumption}
 \theoremstyle{definition}
 \newtheorem{defn}[thm]{Definition}
 \theoremstyle{remark}
 \newtheorem{rem}[thm]{Remark}
 
 \numberwithin{equation}{section}

\usepackage{epsfig} 
\usepackage{psfrag}
\usepackage{amsmath} 
\usepackage{amssymb}  

\usepackage{url}

\usepackage{enumerate}

\newcommand{\overbar}[1]{\mkern 1.5mu\overline{\mkern-1.5mu#1\mkern-1.5mu}\mkern 1.5mu}
\newcommand{\reals}{\mathbb{R}}

\newcommand{\hilbert}{\mathcal{H}}
\newcommand{\id}{{\rm{Id}}}
\DeclareMathOperator*{\argmin}{argmin}

\usepackage{tikz}
\usepackage{pgfplots}
\usetikzlibrary{calc}
\usetikzlibrary{decorations.pathreplacing}

\begin{document}

\title[Linear Convergence for Douglas-Rachford Splitting]{Tight Global Linear Convergence Rate Bounds for Douglas-Rachford Splitting}

\author[Pontus Giselsson]{Pontus Giselsson}

\address{%
Department of Automatic Control\\
Box. 118\\
SE-22100 Lund\\
Sweden}

\email{pontusg@control.lth.se}

\thanks{This project is financially supported by the Swedish Foundation for Strategic Research.}



\begin{abstract}

Recently, several authors have shown local and global convergence rate results for Douglas-Rachford splitting under strong monotonicity, Lipschitz continuity, and cocoercivity assumptions. Most of these focus on the convex optimization setting. In the more general monotone inclusion setting, Lions and Mercier showed a linear convergence rate bound under the assumption that one of the two operators is strongly monotone and Lipschitz continuous. We show that this bound is not tight, meaning that no problem from the considered class converges exactly with that rate. In this paper, we present tight global linear convergence rate bounds for that class of problems. We also provide tight linear convergence rate bounds under the assumptions that one of the operators is strongly monotone and cocoercive, and that one of the operators is strongly monotone and the other is cocoercive. All our linear convergence results are obtained by proving the stronger property that the Douglas-Rachford operator is contractive.


\end{abstract}

\maketitle




\section{Introduction}

Douglas-Rachford splitting \cite{DouglasRachford,LionsMercier1979}
is an algorithm that solves monotone
inclusion problems of the form
\begin{align*}
0\in Ax+Bx
\end{align*}
where $A$ and $B$ are maximally monotone operators. A class of problems
that falls under this category is composite convex optimization
problems of the form
\begin{align}
\begin{tabular}{ll}
minimize & $f(x)+g(x)$
\end{tabular}
\label{eq:comp_prob}
\end{align}
where $f$ and $g$ are proper, closed, and convex functions. This holds
 since the subdifferential of proper, closed, and convex functions
are maximally monotone operators, and since Fermat's rule says that
the optimality condition 
for solving \eqref{eq:comp_prob} is $0\in\partial f(x)+\partial g(x)$, under a suitable qualification condition.
The algorithm has shown great potential in many applications
such as signal processing \cite{Combettes2011}, image denoising
\cite{Setzer2009}, and statistical estimation \cite{BoydDistributed}
(where the dual algorithm ADMM is discussed).

It has long been known that Douglas-Rachford splitting converges under
quite mild assumptions, see
\cite{Gabay1976,LionsMercier1979,EcksteinPhD}. 
However, the rate of convergence in
the general case has just recently been shown to be $O(1/\sqrt{k})$ for the fixed-point residual,
\cite{DR_one_over_k_2012,even_lin_conv_2013,conv_split_schemes_Davis_2014}.
For general maximal monotone operator problems, where one of the
operators is strongly monotone and Lipschitz continuous, Lions and Mercier showed in
\cite{LionsMercier1979} that the Douglas-Rachford algorithm enjoys a
linear convergence rate. To the author's knowledge, this was the sole
linear convergence rate results for a long period of time for these methods. Recently,
however, many works have shown linear convergence rates for
Douglas-Rachford splitting and its dual version, ADMM, see, 
\cite{HesseLuke2013,HesseLuke2014,Phan2014,even_lin_conv_2013,Davis_Yin_2014,linConvADMM,GhadimiADMM, Panos_acc_DR_2014,lin_conv_DR_mult_block_2013,exp_conv_dist_ADMM_2013,boley_2013,Shi_lin_conv_ADMM_2014,gisBoydTAC2014metric_select,gisCDC2015,Bauschke_lin_rate_Friedrich,Bauschke_opt_rate_matr}.
The works in \cite{HesseLuke2013,HesseLuke2014,even_lin_conv_2013,boley_2013,Phan2014} concern local
linear convergence under different assumptions. The works in
\cite{lin_conv_DR_mult_block_2013,exp_conv_dist_ADMM_2013,Shi_lin_conv_ADMM_2014}
consider distributed formulations,
while the works in
\cite{Davis_Yin_2014,linConvADMM,GhadimiADMM,Panos_acc_DR_2014,LionsMercier1979,arvind_ADMM,laurent_ADMM,gisBoydTAC2014metric_select,gisCDC2015,Bauschke_lin_rate_Friedrich,Bauschke_opt_rate_matr}
show global convergence rate bounds under various assumptions. Of
these, the works in 
\cite{gisCDC2015,Bauschke_lin_rate_Friedrich,Bauschke_opt_rate_matr}
show tight linear convergence rate bounds. 
The works in \cite{Bauschke_lin_rate_Friedrich,Bauschke_opt_rate_matr} show
tight convergence rate results for problem of finding a point in the
intersection of two subspaces. In \cite{gisCDC2015} it is shown
that the linear convergence rate bounds in
\cite{gisBoydTAC2014metric_select} (which are generalizations of the
bounds in \cite{GhadimiADMM}) are tight for composite convex
optimization problems where one function is strongly convex and
smooth. All these results, except the one by Lions and Mercier, are stated in the convex optimization setting. In this paper, we will provide tight linear convergence rate bounds for monotone inclusion problems.

We consider three different sets of assumptions under which we provide linear convergence rate bounds. In all cases, the properties of Lipschitz continuity or cocoercivity, and strong monotonicity, are attributed to the operators. In the first case, we assume that one operator is strongly monotone and the other is cocoercive. In the second case, we assume that one operator is both strongly monotone and Lipschitz continuous. This is the setting considered by Lions and Mercier in \cite{LionsMercier1979}, where a non-tight linear convergence rate bound is presented. In the third case, we assume that one operator is both strongly monotone and cocoercive. We show in all these settings that our bounds are tight, meaning that there exists problems from the respective classes that converge exactly with the provided rate bound. In the second and third cases, the rates are tight for all feasible algorithm parameters, while in the first case, the rate is tight for many algorithm parameters.

\section{Background}

In this section, we introduce some notation and define some operator and
function properties.

\subsection{Notation}

We denote by $\reals$ the set of real numbers and by
$\overbar{\reals}:=\reals\cup\{\infty\}$ the extended
real line. 
Throughout this paper, $\hilbert$ denotes a separable real Hilbert
space. Its inner product is denoted by $\langle\cdot,\cdot\rangle$, its
induced norm by $\|\cdot\|$. 
We denote by $\{\phi_i\}_{i=1}^K$ any orthonormal basis in $\hilbert$,
where $K$ is the dimension of $\hilbert$ (possibly $\infty$).
The gradient to $f~:~\mathcal{X}\to\reals$ is denoted by $\nabla f$ and 
the subdifferential operator to $f~:~\mathcal{X}\to\overbar{\reals}$ is denoted by $\partial f$ and is defined as $\partial f(x_1):=\{u~|~f(x_2)\geq f(x_1)+\langle u,x_2-x_1\rangle{\hbox{ for all }} x_2\in\mathcal{X}\}$. 
The
conjugate function of $f$ is denoted and defined by 
$f^{*}(y)\triangleq \sup_{x}\left\{\langle y,x\rangle-f(x)\right\}$.
The power set of a set $\mathcal{X}$, i.e., the set of all subsets of
$\mathcal{X}$, is denoted by $2^{\mathcal{X}}$. The graph of an
(set-valued) operator $A~:~\mathcal{X}\to 2^{\mathcal{Y}}$ is defined and denoted
by ${\rm{gph}}A = \{(x,y)\in\mathcal{X}\times\mathcal{Y}~|~y\in Ax\}$.
The inverse operator $A^{-1}$ is defined through its graph by
${\rm{gph}}A^{-1} = \{(y,x)\in\mathcal{Y}\times\mathcal{X}~|~y\in
Ax\}$. The identity operator is denoted by $\id$ and the resolvent of a monotone operator $A$ is defined and denoted by $J_A=(\id+A)^{-1}$. Finally, the
class of closed, proper, and convex
functions $f~:~\hilbert\to\overbar{\reals}$ is denoted by
$\Gamma_0(\hilbert)$. 

\subsection{Operator properties}

\begin{defn}[Strong monotonicity]
  Let $\sigma> 0$. An operator $A~:~\hilbert\to 2^{\hilbert}$ is
  $\sigma$-\emph{strongly monotone} if
  \begin{align*}
    \langle u-v,x-y\rangle\geq \sigma\|x-y\|^2
  \end{align*}
  holds for all $(x,u)\in{\rm{gph}}(A)$ and $(y,v)\in{\rm{gph}}(A)$.
  \label{def:str_monotone}
\end{defn}

The operator is merely \emph{monotone} if $\sigma=0$ in the above definition.
In the following three definitions, we
state some properties for single-valued operators $T~:~\hilbert\to\hilbert$. We state the properties for operators with full domain, but they can also be stated for operators with any nonempty domain $\mathcal{D}\subseteq\hilbert$.

\begin{defn}[Lipschitz continuity]
  Let $\beta\geq 0$. A mapping $T~:~\hilbert\to\hilbert$ is
  $\beta$-\emph{Lipschitz continuous} if
  \begin{align*}
    \|Tx-Ty\|\leq \beta\|x-y\|
  \end{align*}
  holds for all $x,y\in\hilbert$. 
  \label{def:Lipschitz}
\end{defn}

\begin{defn}[Nonexpansiveness]
  A mapping $T~:~\hilbert\to\hilbert$ is
  \emph{nonexpansive} if it is $1$-Lipschitz continuous.
  \label{def:nonexpansive}
\end{defn}

\begin{defn}[Contractiveness]
  A mapping $T~:~\hilbert\to\hilbert$ is
  $\delta$-\emph{contractive} if it is $\delta$-Lipschitz continuous with $\delta\in[0,1)$.
  \label{def:contractive}
\end{defn}

\begin{defn}[Averaged mappings]
  A mapping $T~:~\hilbert\to\hilbert$ is
  $\alpha$-\emph{averaged} if there exists a nonexpansive mapping
  $R~:~\hilbert\to\hilbert$
  and $\alpha\in(0,1)$ such that $T=(1-\alpha)\id+\alpha R$.
\label{def:averaged}
\end{defn}

From \cite[Proposition~4.25]{bauschkeCVXanal}, we know that an operator $T~:~\hilbert\to\hilbert$ is
$\alpha$-averaged if and only if it satisfies
\begin{align}
\tfrac{1-\alpha}{\alpha}\|(\id-T)x-(\id-T)y\|^2+\|Tx-Ty\|^2\leq\|x-y\|^2
\label{eq:averaged}
\end{align}
for all $x,y\in\hilbert$.

\begin{defn}[Cocoercivity]
  Let $\beta> 0$. A mapping $T~:~\hilbert\to\hilbert$ is $\tfrac{1}{\beta}$-cocoercive if
  \begin{align*}
    \langle Tx-Ty,x-y\rangle\geq \tfrac{1}{\beta}\|Tx-Ty\|^2
  \end{align*}
  holds for all $x,y\in\hilbert$.
  \label{def:cocoercive}
\end{defn}

\section{Preliminaries}

In this section, we state and show preliminary results that are needed to
prove the linear convergence rate bounds. We state some lemmas
that describe how cocoercivity, Lipschitz continuity, as well as
averagedness relate to each other. We also introduce
\emph{negatively averaged operators}, $T$, that are defined by that $-T$ is
averaged. We show different properties of such operators,
including that averaged maps of negatively averaged operators are contractive. This result will be used to show linear convergence in the case where the strong monotonicity and Lipschitz continuity properties are split between the operators.

\subsection{Useful lemmas}

\label{sec:operator_prop_rel}

Proofs to the following three lemmas are found in Appendix~\ref{app:lemma_proofs}.

\begin{lem}
Assume that $\beta>0$ and let $T~:~\hilbert\to\hilbert$. Then $\tfrac{1}{2\beta}$-cocoercivity of
$\beta\id+T$ is equivalent to $\beta$-Lipschitz continuity of $T$.
\label{lem:coco_vs_lipschitz}
\end{lem}

\begin{lem}
Assume that $\beta\in(0,1)$. Then
$\tfrac{1}{\beta}$-cocoercivity of $R~:~\hilbert\to\hilbert$ is equivalent to
$\tfrac{\beta}{2}$-averagedness of $T=R+(1-\beta)\id$.
\label{lem:coco_averaged}
\end{lem}

\begin{lem}
Let $T~:~\hilbert\to\hilbert$ be $\delta$-contractive with $\delta\in[0,1)$. Then $R=(1-\alpha)\id+\alpha T$ is contractive for all $\alpha\in(0,\tfrac{2}{1+\delta})$. The contraction factor is $|1-\alpha|+\alpha\delta$.
\label{lem:avg_contr}
\end{lem}

For easier reference, we also record special cases of some results in \cite{bauschkeCVXanal} that will be used later. Specifically, we record, in order, special cases of \cite[Proposition~4.33]{bauschkeCVXanal}, \cite[Proposition~4.28]{bauschkeCVXanal}, and \cite[Proposition~23.11]{bauschkeCVXanal}.

\begin{lem}
Let $\beta\in(0,1)$ and let $T~:~\hilbert\to\hilbert$ be $\tfrac{1}{\beta}$-cocoercive. Then $(\id-T)$ is $\tfrac{\beta}{2}$-averaged.
\label{lem:coco_averaged2}
\end{lem}

\begin{lem}
Let $T~:~\hilbert\to\hilbert$ be $\alpha$-averaged with $\alpha\in(0,\tfrac{1}{2})$. Then $(2T-\id)$ is $2\alpha$-averaged.
\label{lem:res_reflres_ave}
\end{lem}

\begin{lem}
Let $A~:~\hilbert\to 2^{\hilbert}$ be maximally monotone and $\sigma$-strongly monotone with $\sigma>0$. Then $J_A=(\id+A)^{-1}$ is $(1+\sigma)$-cocoercive.
\label{lem:res_coco}
\end{lem}

\subsection{Negatively averaged operators}

In this section we define negatively averaged operators and show various
properties for these.
\begin{defn}
An operator $T~:~\hilbert\to\hilbert$ is $\theta$-negatively averaged with $\theta\in(0,1)$ if
$-T$ is $\theta$-averaged.
\end{defn}

This definition implies that an operator $T$ is $\theta$-negatively averaged if and only if
it satisfies
\begin{align*}
T = -((1-\theta)\id+\theta \bar{R})=(\theta-1)\id+\theta R
\end{align*}
where $\bar{R}$ is nonexpansive and $R:=-\bar{R}$ is therefore also
nonexpansive. Since $-T$ is averaged, it is also nonexpansive,
and so is $T$.

Since negatively averaged operators are nonexpansive, 
they can be averaged. 
\begin{defn}
An $\alpha$-averaged $\theta$-negatively averaged operator $S~:~\hilbert\to\hilbert$ is defined as
$S=(1-\alpha)\id+\alpha 
T$ where $T~:~\hilbert\to\hilbert$ is $\theta$-negatively averaged.
\end{defn}

Next, we show that averaged negatively averaged operators are contractive.
\begin{prop}
An $\alpha$-averaged $\theta$-negatively averaged operator $S~:~\hilbert\to\hilbert$ is
$|1-2\alpha+\alpha\theta|+\alpha\theta$-contractive.
\label{prp:avg_anti_avg_contr}
\end{prop}
\begin{proof}
Let $T=(\theta-1)\id+\theta R$ (for some nonexpansive $R$) be the
$\theta$-negatively averaged operator, which implies that
$S=(1-\alpha)\id+\alpha T$. Then
\begin{align*}
\|Sx-Sy\|&=\|((1-\alpha)\id+\alpha T)x-((1-\alpha)\id+\alpha T)y\|\\
&= \|(1-2\alpha+\alpha\theta)(x-y)+\alpha\theta (Rx-Ry))\|\\
&\leq |1-2\alpha+\alpha\theta|\|x-y\|+\alpha\theta\|x-y\|\\
&= (|1-2\alpha+\alpha\theta|+\alpha\theta)\|x-y\|
\end{align*}
since $R$ is nonexpansive. It is straightforward to verify that
$|1-2\theta+\alpha\theta|+\alpha\theta<1$ for all combinations of
$\alpha\in(0,1)$ and $\theta\in(0,1)$. Hence, $S$ is contractive and
the proof is complete.
\end{proof}
Next, we optimize the contraction factor w.r.t. $\alpha$.
\begin{prop}
Assume that $T~:~\hilbert\to\hilbert$ is $\theta$-negatively averaged. Then the $\alpha$ that optimizes
the contraction factor for the $\alpha$-averaged $\theta$-negatively averaged
operator $S=(1-\alpha)\id+\alpha T$ is $\alpha = \tfrac{1}{2-\theta}$.
The corresponding optimal contraction constant is $\tfrac{\theta}{2-\theta}$.
\label{prp:avg_antiavg_optparam_subdiff}
\end{prop}
\begin{proof}
Due to the absolute value, Proposition~\ref{prp:avg_anti_avg_contr}
states that the contraction factor $\delta$ of $T$ can be written as
\begin{align*}
\delta&=\begin{cases}
1-2\alpha+\alpha\theta+\alpha\theta & {\hbox{if }} \alpha\leq
\tfrac{1}{2-\theta}\\
-(1-2\alpha+\alpha\theta)+\alpha\theta & {\hbox{if }} \alpha\geq
\tfrac{1}{2-\theta}\\
\end{cases}
=\begin{cases}
1-2(1-\theta)\alpha & {\hbox{if }} \alpha\leq
\tfrac{1}{2-\theta}\\
2\alpha-1 & {\hbox{if }} \alpha\geq
\tfrac{1}{2-\theta}
\end{cases}
\end{align*}
where the kink in the absolute value term is at
$\alpha=\tfrac{1}{2-\theta}$. Since $\theta\in(0,1)$, we get negative slope for
$\alpha\leq\tfrac{1}{2-\theta}$ and positive slope for
$\alpha\geq\tfrac{1}{2-\theta}$. Therefore the optimal $\alpha$ is in
the kink at $\alpha=\tfrac{1}{2-\theta}$, which satisfies $\alpha\in(\tfrac{1}{2},1)$ since
$\theta\in(0,1)$. Inserting this into the contraction factor
expression gives $\tfrac{\theta}{2-\theta}$. This concludes the proof.
\end{proof}
\begin{rem}
The optimal contraction factor $\tfrac{\theta}{2-\theta}$ is strictly increasing
in $\theta$ on the interval 
$\theta\in(0,1)$. Therefore the contraction factor becomes smaller the
smaller $\theta$ is.
\end{rem}

We conclude this section by showing that the composition of an
averaged and a negatively averaged
operator is negatively averaged. 
Before we state the result, we need a characterization of
$\theta$-negatively averaged operators $T$. This
follows directly from the definition of averaged operators in
\eqref{eq:averaged} since $-T$ is $\theta$-averaged:
\begin{align}
\tfrac{1-\theta}{\theta}\|(\id+T)x-(\id+T)y\|^2+\|Tx-Ty\|^2\leq\|x-y\|^2.
\label{eq:anti_averaged}
\end{align}

\begin{prop}
Assume that $T_{\theta}~:~\hilbert\to\hilbert$ is $\theta$-negatively averaged and $T_{\alpha}~:~\hilbert\to\hilbert$ is $\alpha$-averaged.
Then $T_{\theta}T_{\alpha}$ is $\tfrac{\kappa}{\kappa+1}$-negatively averaged where $\kappa
=\tfrac{\theta}{1-\theta}+\tfrac{\alpha}{1-\alpha}$.
\label{prp:avg_antiavg_comp}
\end{prop}
\begin{proof}
Let $\kappa_{\theta}=\tfrac{\theta}{1-\theta}$ and
$\kappa_{\alpha}=\tfrac{\alpha}{1-\alpha}$, then $\kappa =
\kappa_{\theta}+\kappa_{\alpha}$. We have
\begin{align}
\nonumber \|(\id+T_{\theta}T_{\alpha}&)x-(\id+T_{\theta} T_{\alpha})y\|^2\\
\nonumber&=\|(x-y)-(T_{\alpha} x-T_{\alpha} y)+(T_{\alpha}x-T_{\alpha}y)+(T_{\theta}T_{\alpha}x-T_{\theta}T_{\alpha}y)\|^2\\
\nonumber &=\|(\id-T_{\alpha})x-(\id-T_{\alpha})y+(\id+T_{\theta})T_{\alpha}x-(\id+T_{\theta})T_{\alpha}y\|^2\\
\nonumber
&\leq\tfrac{\kappa_{\theta}+\kappa_{\alpha}}{\kappa_{\alpha}}\|(\id-T_{\alpha})x-(\id-T_{\alpha})y\|^2\\
\nonumber&\quad+\tfrac{\kappa_{\theta}+\kappa_{\alpha}}{\kappa_{\theta}}\|(\id+T_{\theta})T_{\alpha}x-(\id+T_{\theta})T_{\alpha}y\|^2\\
\nonumber &\leq (\kappa_{\theta}+\kappa_{\alpha}) (\|x-y\|^2-\|T_{\alpha}x-T_{\alpha}y\|^2)\\
\nonumber &\quad+(\kappa_{\theta}+\kappa_{\alpha}) (\|T_{\alpha}x-T_{\alpha}y\|^2-\|T_{\theta}T_{\alpha}x-T_{\theta}T_{\alpha}y\|^2)\\
\nonumber &=(\kappa_{\theta}+\kappa_{\alpha})(\|x-y\|^2-\|T_{\theta}T_{\alpha}x-T_{\theta}T_{\alpha}y\|^2)\\
&=\kappa(\|x-y\|^2-\|T_{\theta}T_{\alpha}x-T_{\theta}T_{\alpha}y\|^2)
\label{eq:avg_antiavg_comp}
\end{align}
where the first inequality follows from convexity of $\|\cdot\|^2$.
More precisely, let $t\in[0,1]$, then, by convexity of $\|\cdot\|^2$, we conclude that
\begin{align*}
\|x+y\|^2&=\|t\tfrac{1}{t}x+(1-t)\tfrac{1}{1-t}y\|^2\leq t\|\tfrac{1}{t}x\|^2+(1-t)\tfrac{1}{1-t}\|y\|^2\\
&=\tfrac{1}{t}\|x\|^2+\tfrac{1}{1-t}\|y\|^2.
\end{align*}
Letting $t=\tfrac{\kappa_{\alpha}}{\kappa_{\theta}+\kappa_{\alpha}}\in[0,1]$,
which implies that $1-t=\tfrac{\kappa_{\theta}}{\kappa_{\theta}+\kappa_{\alpha}}\in[0,1]$, gives the
first inequality in \eqref{eq:avg_antiavg_comp}. The
second inequality in \eqref{eq:avg_antiavg_comp} follows from
\eqref{eq:averaged} and \eqref{eq:anti_averaged}. The relation in
\eqref{eq:avg_antiavg_comp} coincides with the definition of
negative averagedness in \eqref{eq:anti_averaged}. Thus $T_{\theta}T_{\alpha}$ is
$\phi$-negatively averaged
with $\phi$ satisfying
$\tfrac{1-\phi}{\phi}=\tfrac{1}{\kappa}$. This gives $\phi =
\tfrac{\kappa}{\kappa+1}$ and the proof is complete.
\end{proof}
\begin{rem}
This result can readily be extended to show averagedness of
$T=T_1T_2\cdots T_N$ where $T_i$ are $\alpha_i$-(negatively) averaged for
$i=1,\ldots,N$. We get that $T$ is
$\tfrac{\kappa}{1+\kappa}$-negatively averaged with
$\kappa=\sum_{i=1}^N\tfrac{\alpha_i}{1-\alpha_i}$ if an odd number of the $T_i$:s
are negatively averaged, and that $T$ is $\tfrac{\kappa}{1+\kappa}$-averaged
if an even number of the $T_i$ are negatively averaged.
Similar results have been and presented, e.g., in
\cite[Proposition~4.32]{bauschkeCVXanal} which is improved in \cite{Combettes201555}. Our result extends these results in that it allows also for
negatively averaged operators and reduces to the result in \cite{Combettes201555} for averaged operators.
\end{rem}

\section{Douglas-Rachford splitting}

Douglas-Rachford splitting can be applied to solve monotone inclusion
problems of the form
\begin{align}
0\in Ax+Bx
\label{eq:monotone_inclusion}
\end{align}
where $A,B~:~\hilbert\to 2^{\hilbert}$ are maximally monotone operators. The algorithm 
separates $A$ and $B$ by only touching the corresponding \emph{resolvents}, where the
resolvent $J_A~:~\hilbert\to\hilbert$ is defined as
\begin{align*}
J_{A} := (A+\id)^{-1}.
\end{align*}
The resolvent has full domain since $A$ is assumed maximally monotone, see \cite{minty1962} and \cite[Proposition~23.7]{bauschkeCVXanal}.
If $A=\partial f$ where $f$ is a proper, closed, and convex function,
then $J_A = {\rm{prox}}_{f}$ where the prox operator ${\rm{prox}}_{f}$ is defined as
\begin{align}
{\rm{prox}}_{f}(z) = \argmin_x\left\{f(x)+\tfrac{1}{2}\|x-z\|^2\right\}.
\label{eq:prox_def}
\end{align}
That this holds follows directly from Fermat's rule
\cite[Theorem~16.2]{bauschkeCVXanal} applied to 
the proximal operator definition.

The Douglas-Rachford algorithm is defined by the iteration
\begin{align}
z^{k+1} = ((1-\alpha)\id+\alpha R_AR_B)z^{k}
\label{eq:DR_alg}
\end{align}
where $\alpha\in(0,1)$ (we will see that also $\alpha\geq 1$ can sometimes
be used) and $R_A~:~\hilbert\to\hilbert$ is the \emph{reflected resolvent},
which is defined as 
\begin{align*}
R_A:=2J_A-\id.
\end{align*}
(Note that what is traditionally called 
Douglas-Rachford splitting is when $\alpha=1/2$ in \eqref{eq:DR_alg}. The case with
$\alpha=1$ in \eqref{eq:DR_alg} is often referred to as the Peaceman-Rachford
algorithm, see \cite{PeacemanRachford}. We 
will use the term Douglas-Rachford splitting for all feasible choices of  
$\alpha$.)

Since the reflected resolvent is nonexpansive in
the general case \cite[Corollary~23.10]{bauschkeCVXanal}, and since
compositions of nonexpansive operators
are nonexpansive, the Douglas-Rachford algorithm is an averaged
iteration of a nonexpansive mapping when $\alpha\in(0,1)$. Therefore,
Douglas-Rachford splitting is a special case of the Krasnosel'ski\u{\i}-Mann
iteration \cite{Mann_1953,Krasnoselskii_1955}, which is known to converge to a fixed-point of the
nonexpansive operator, in this case $R_AR_B$, see \cite[Theorem~5.14]{bauschkeCVXanal}. Since an $x\in\hilbert$
solves \eqref{eq:monotone_inclusion} if and only if $x=J_Az$ where
$z=R_AR_Bz$, see \cite[Proposition~25.1]{bauschkeCVXanal} this
algorithm can be used to solve monotone inclusion problems of the form
\eqref{eq:monotone_inclusion}.
Note that to solve \eqref{eq:monotone_inclusion}, is equivalent to solving
\begin{align*}
0\in\gamma Ax+\gamma Bx
\end{align*}
for any $\gamma\in(0,\infty)$. Then we can define $A_{\gamma} =
\gamma A$ and \eqref{eq:monotone_inclusion} 
can also be solved by the iteration
\begin{align}
z^{k+1} = ((1-\alpha)\id+\alpha R_{A_\gamma}R_{B_\gamma})z^{k}.
\label{eq:DR_alg_gen}
\end{align}
Therefore, $\gamma$ is an algorithm parameter that affects the
progress of the iterations.

The objective of this paper, is to provide tight linear convergence rate bounds for the Douglas-Rachford algorithm under various assumptions. Using these bounds, we will show how to select the algorithm parameters $\gamma$ and $\alpha$ that optimize these bounds. The first setting we consider is when $A$ is strongly monotone and $B$ is cocoercive.

\section{$A$ strongly monotone and $B$ cocoercive}
\label{sec:split_prop}

In this section, we show linear convergence for Douglas-Rachford splitting in the case where $A$ and $B$ are maximally monotone, $A$ is strongly monotone, and $B$ is cocoercive. That is, we make the following assumptions.
\begin{ass}
Suppose that:
\begin{enumerate}[(i)]
\item $A~:~\hilbert\to 2^{\hilbert}$ is maximally monotone and $\sigma$-strongly monotone.
\item $B~:~\hilbert\to\hilbert$ is maximally monotone and $\tfrac{1}{\beta}$-cocoercive.
\end{enumerate}
\label{ass:subdiff}
\end{ass}

Before we can state the main linear convergence result, we need to characterize the properties of the resolvent, the reflected resolvent, and the composition between reflected resolvents. This is done in the following series of propositions, this first of which is proven in Appendix~\ref{app:pf_split_prop}.

\begin{prop}
The resolvent $J_B$ of a $\tfrac{1}{\beta}$-cocoercive 
operator $B~:~\hilbert\to\hilbert$ is
$\tfrac{\beta}{2(1+\beta)}$-averaged.
\label{prp:res_prop_lipschitz_subdiff}
\end{prop}

This implies that also the reflected resolvent is averaged.
\begin{prop}
The reflected resolvent of a $\tfrac{1}{\beta}$-cocoercive
operator $B~:~\hilbert\to\hilbert$ is
$\tfrac{\beta}{1+\beta}$-averaged. 
\label{prp:refl_res_prop_lipschitz_subdiff}
\end{prop}
\begin{proof}
This follows directly from the
Proposition~\ref{prp:res_prop_lipschitz_subdiff} and
Lemma~\ref{lem:res_reflres_ave}.
\end{proof}
If the operator instead is strongly monotone, the reflected resolvent is negatively averaged.
\begin{prop}
  The reflected resolvent of a $\sigma$-strongly monotone and maximal monotone operator
$A~:~\hilbert\to 2^\hilbert$ is $\tfrac{1}{1+\sigma}$-negatively averaged.
\label{prp:refl_res_prop_strmono_subdiff}
\end{prop}
\begin{proof}
From Lemma~\ref{lem:res_coco}, we have that
the resolvent $J_A$ is $(1+\sigma)$-cocoercive. Using Lemma~\ref{lem:coco_averaged2}, this implies that $\id-J_A$ is $\tfrac{1}{2(1+\sigma)}$-averaged. Then using Lemma~\ref{lem:res_reflres_ave}, this implies that $2(\id-J_A)-\id=\id-2J_A=-R_A$ is $\tfrac{1}{1+\sigma}$-averaged, hence $R_A$ is $\tfrac{1}{1+\sigma}$-negatively averaged. This completes the proof.
\end{proof}
The composition of the reflected resolvents of a
strongly monotone operator and a cocoercive operator
is negatively averaged.
\begin{prop}
Suppose that Assumption~\ref{ass:subdiff} holds. Then, the composition $R_AR_B$ is
$\tfrac{\tfrac{1}{\sigma}+\beta}{1+\tfrac{1}{\sigma}+\beta}$-negatively
averaged.
\label{prp:comp_refl_res_anti_averaged_subdiff}
\end{prop}
\begin{proof}
Since $R_A$ is $\tfrac{1}{1+\sigma}$-negatively averaged and $R_B$ is
$\tfrac{\beta}{1+\beta}$-averaged, see
Propositions~\ref{prp:refl_res_prop_lipschitz_subdiff} and
\ref{prp:refl_res_prop_strmono_subdiff}, we can apply
Proposition~\ref{prp:avg_antiavg_comp}. We get that 
$\kappa = 
\tfrac{\tfrac{1}{1+\sigma}}{1-\tfrac{1}{1+\sigma}}+\tfrac{\tfrac{\beta}{1+\beta}}{1-\tfrac{\beta}{1+\beta}}=\tfrac{1}{\sigma}+\beta$
and that the averagedness parameter of the negatively averaged operator $R_AR_B$ is given by
$\tfrac{\kappa}{\kappa+1}=\tfrac{\tfrac{1}{\sigma}+\beta}{\tfrac{1}{\sigma}+\beta+1}$.
This concludes the proof.
\end{proof}

With these results, we can now show the following linear convergence rate bounds for Douglas-Rachford splitting under Assumption~\ref{ass:subdiff}. The theorem is proven in Appendix~\ref{app:pf_split_prop}.
\begin{thm}
Suppose that Assumption~\ref{ass:subdiff} holds, that $\alpha\in(0,1)$, that $\gamma\in(0,\infty)$, and that the Douglas-Rachford algorithm \eqref{eq:DR_alg} is applied to solve $0\in \gamma Ax+\gamma Bx$. Then the algorithm converges at least with rate factor
\begin{align}
\left|1-2\alpha+\alpha\tfrac{\tfrac{1}{\gamma\sigma}+\gamma\beta}{1+\tfrac{1}{\gamma\sigma}+\gamma\beta}\right|+\alpha\tfrac{\tfrac{1}{\gamma\sigma}+\gamma\beta}{1+\tfrac{1}{\gamma\sigma}+\gamma\beta}.
\label{eq:DR_rate_bound_AB_prop_subdiff}
\end{align}
Optimizing this rate bound w.r.t. $\alpha$ and $\gamma$ gives $\gamma
= \tfrac{1}{\sqrt{\beta\sigma}}$ and
$\alpha=\tfrac{\sqrt{\beta/\sigma}+1/2}{1+\sqrt{\beta/\sigma}}$. The
corresponding optimal rate bound is
$\tfrac{\sqrt{\beta/\sigma}}{\sqrt{\beta/\sigma}+1}$. 
\label{thm:rate_subdiff}
\end{thm}

\subsection{Tightness}

In this section, we present an example that shows tightness of the linear convergence rate bounds in Theorem~\ref{thm:rate_subdiff} for many algorithm parameters. We consider a two dimensional Euclidean example, which is given by the following convex optimization problem:
\begin{align}
\begin{tabular}{ll}
minimize & $f(x)+f^*(x)$
\end{tabular}
\label{eq:tight_example_subdiff}
\end{align}
where 
\begin{align}
f(x)=\tfrac{\beta}{2}x_1^2,
\label{eq:f_def}
\end{align} 
and $x=(x_1,x_2)$, and $\beta>0$. The gradient $\nabla f=\beta x_1$, so it is cocoercive with factor $\tfrac{1}{\beta}$. According to \cite[Theorem~18.15]{bauschkeCVXanal} this is equivalent to that $f^*$ is $\tfrac{1}{\beta}$-strongly convex and therefore $\partial f^*$ is $\sigma:=\tfrac{1}{\beta}$-strongly monotone.

The following proposition shows that when solving \eqref{eq:tight_example_subdiff} with $f$ defined in \eqref{eq:f_def} using Douglas-Rachford splitting, the upper linear convergence rate bound is exactly attained. The result is proven in Appendix~\ref{app:pf_split_prop}.
\begin{prop}
Suppose that the Douglas-Rachford algorithm \eqref{eq:DR_alg} is applied to solve \eqref{eq:tight_example_subdiff} with $f$ in \eqref{eq:f_def}. Further suppose that the parameters $\gamma$ and $\alpha$ satisfy $\gamma\in(0,\infty)$ and $\alpha\in[c,1)$ where $c=\tfrac{1+\gamma\sigma+\gamma^2\sigma\beta}{1+2\gamma\sigma+\gamma^2\sigma\beta}$ and that $z^0=(0,z_2^0)$ with $z_2^0\neq 0$. Then the $z^k$ sequence in \eqref{eq:DR_alg} converges exactly with rate \eqref{eq:DR_rate_bound_AB_prop_subdiff} in Theorem~\ref{thm:rate_subdiff}.
\label{prp:tight_split_prop}
\end{prop}

So, for all $\gamma$ parameters and some $\alpha$ parameters, the provided bound is tight. Especially, the optimal parameter choices
$\gamma=\tfrac{1}{\sqrt{\beta\sigma}}$ and $\alpha=\tfrac{1+2\sqrt{\beta/\sigma}}{2(1+\sqrt{\beta/\sigma})}$ give a tight bound.

It is interesting to note that although we have considered a more general class of problems than convex optimization problems, a convex optimization problem is used to attain the worst case rate. 


\subsection{Comparison to other bounds}


In \cite{gisBoydTAC2014metric_select}, it was shown that Douglas-Rachford splitting converges as $\tfrac{\sqrt{\beta/\sigma}-1}{\sqrt{\beta/\sigma}+1}$
when solving composite optimization 
problems of the form $0\in\gamma\nabla f+\gamma\partial g$, where
$\nabla f$ is $\sigma$-strongly monotone and $\tfrac{1}{\beta}$-cocoercive and the algorithm
parameters are chosen as $\alpha=1$ and $\gamma = \tfrac{1}{\sqrt{\beta\sigma}}$. In our setting,
with $\partial f$ being $\sigma$-strongly monotone and
$\partial g$ being $\tfrac{1}{\beta}$-cocoercive, we can instead pose the equivalent problem
$0\in\gamma\partial\hat{f}(x)+\gamma\partial\hat{g}(x)$ where $\hat{f} =
f-\tfrac{\sigma}{2}\|\cdot\|^2$ and
$\hat{g}=g+\tfrac{\sigma}{2}\|\cdot\|^2$. Then $\partial \hat{f}$ is merely monotone and
$\hat{g}$ is $\sigma$-strongly monotone and $\tfrac{1}{\beta+\sigma}$-cocoercive. For that problem,
\cite{gisBoydTAC2014metric_select} shows a linear convergence rate of at least rate $\tfrac{\sqrt{(\beta+\sigma)/\sigma}-1}{\sqrt{(\beta+\sigma)/\sigma}+1}$
(when optimal parameters are used).
This rate turns out to be better than the rate provided in
Theorem~\ref{thm:rate_subdiff}, i.e.
$\tfrac{\sqrt{\beta/\sigma}}{\sqrt{\beta/\sigma}+1}$, which assumes
that the strong convexity and smoothness properties are split between the
two operators. This is shown by the following chain of equivalences which departs from the fact that the square root is sub-additive, i.e., that $\sqrt{\beta+\sigma}\leq \sqrt{\sigma}+\sqrt{\beta}$ for $\beta,\sigma\geq 0$:
\begin{align*}
&&\sqrt{\beta+\sigma}-\sqrt{\sigma}&\leq \sqrt{\beta}\\
&\Leftrightarrow& \sqrt{(\beta+\sigma)/\sigma}-1&\leq \sqrt{\beta/\sigma}\\
&\Leftrightarrow& \tfrac{\sqrt{(\beta+\sigma)/\sigma}-1}{\sqrt{(\beta+\sigma)/\sigma}+1}&\leq \tfrac{\sqrt{\beta/\sigma}}{\sqrt{(\beta+\sigma)/\sigma}+1}\left(\leq \tfrac{\sqrt{\beta/\sigma}}{\sqrt{\beta/\sigma}+1}\right)
\end{align*}
This implies that,
from a worst case perspective, it is better to shift both properties into one
operator. This is also always possible, without increasing the
computational cost in the algorithm, since the prox-operator is
just shifted slightly:
\begin{align*}
{\rm{prox}}_{\gamma\hat{f}}(z) &=
\argmin_x\left\{\hat{f}(x)+\tfrac{1}{2\gamma}\|x-z\|^2\right\}\\
&=\argmin_x\left\{f(x)-\tfrac{\sigma}{2}\|x\|^2+\tfrac{1}{2\gamma}\|x-z\|^2\right\}\\
&=\argmin_x\left\{f(x)+\tfrac{1-\gamma\sigma}{2\gamma}\|x-\tfrac{1}{1-\gamma\sigma}z\|^2\right\}\\
&={\rm{prox}}_{\tfrac{\gamma}{1-\gamma\sigma} f}(\tfrac{1}{1-\gamma\sigma}z).
\end{align*}
A similar relation holds for ${\rm{prox}}_{\gamma\hat{g}}$ with the
sign in front of $\gamma\sigma$ flipped.


\section{$A$ strongly monotone and Lipschitz continuous}
\label{sec:A_sm_l}

In this section, we consider the case where one of the
operators is $\sigma$-strongly monotone and $\beta$-Lipschitz
continuous. This is assumption is stated next.
\begin{ass}
Suppose that:
\begin{enumerate}[(i)]
\item The operators $A~:~\hilbert\to\hilbert$ and $B~:~\hilbert\to 2^{\hilbert}$ are maximally monotone.
\item $A$ is $\sigma$-strongly monotone and $\beta$-Lipschitz continuous.
\end{enumerate}
\label{ass:general}
\end{ass}

First, we state a result that characterizes the resolvent of $A$. It is proven in Appendix~\ref{app:pf_A_sm_l}.
\begin{prop}
Assume that $A~:~\hilbert\to \hilbert$ is a maximal monotone $\beta$-Lipschitz continuous
operator. Then the resolvent $J_{A}=(\id+A)^{-1}$ 
satisfies 
\begin{align}
2\langle J_Ax-J_Ay,x-y\rangle&\geq \|x-y\|^2
+(1-\beta^2)\|J_Ax-J_Ay\|^2.
\label{eq:res_prop_A_lipschitz}
\end{align}
\label{prp:res_str_mono_general}
\end{prop}
This resolvent property is used when proving the following contraction factor of the reflected resolvent. The result is proven in Appendix~\ref{app:pf_A_sm_l}.

\begin{thm}
Suppose that $A~:~\hilbert\to \hilbert$ is a $\sigma$-strongly monotone and $\beta$-Lipschitz continuous operator. Then the reflected resolvent $R_{A} = 2J_A-\id$ is $\sqrt{1-\tfrac{4\sigma}{1+2\sigma+\beta^2}}$-contractive.
\label{thm:refl_res_contr}
\end{thm}

The parameter $\gamma$ that optimizes the contraction factor for
$R_{\gamma A}$ is the minimizer of
$h(\gamma):=1-\tfrac{4\gamma\sigma}{1+\gamma\sigma+(\gamma\beta)^2}$ ($\gamma A$ is $\gamma\sigma$-strongly monotone and
$\gamma\beta$-Lipschitz continuous). The gradient $\nabla h(\gamma) =
\tfrac{4\sigma(\beta^2\gamma^2-1)}{(\beta^2\gamma^2+2\sigma\gamma+1)^2}$,
which implies that the extreme points are given by $\gamma=\pm\tfrac{1}{\beta}$. Since
$\gamma>0$ and the gradient is positive for
$\gamma>\tfrac{1}{\beta}$ and negative for $\gamma\in(0,\tfrac{1}{\beta})$,
$\gamma=\tfrac{1}{\beta}$ optimizes the contraction factor.
The corresponding rate is
\begin{align*}
\sqrt{1-\tfrac{4\gamma\sigma}{1+2\gamma\sigma+(\gamma\beta)^2}} = 
\sqrt{1-\tfrac{2\sigma/\beta}{1+\sigma/\beta}}=\sqrt{\tfrac{1-\sigma/\beta}{1+\sigma/\beta}}=\sqrt{\tfrac{\beta/\sigma-1}{\beta/\sigma+1}}.
\end{align*}
This is summarized in the following proposition.
\begin{prop}
The parameter $\gamma$ that optimizes the contraction factor of
$R_{\gamma A}$ is given by $\gamma = \tfrac{1}{\beta}$.
The corresponding contraction factor is $\sqrt{\tfrac{\beta/\sigma-1}{\beta/\sigma+1}}$. 
\label{prp:A_sm_l_optparam}
\end{prop}

Now, we are ready to state the convergence rate results for Douglas-Rachford splitting.
\begin{thm}
Suppose that Assumption~\ref{ass:general} holds and that the
Douglas-Rachford algorithm \eqref{eq:DR_alg} is applied to solve $0\in
\gamma Ax+\gamma Bx$. Let $\delta=\sqrt{1-\tfrac{4\gamma\sigma}{1+2\gamma\sigma+(\gamma\beta)^2}}$, then
the algorithm converges at least with rate factor
\begin{align}
|1-\alpha|+\alpha\delta
\label{eq:rate_factor_general}
\end{align}
for all $\alpha\in(0,\tfrac{2}{1+\delta})$.
Optimizing this bound w.r.t. $\alpha$ and $\gamma$ gives $\alpha=1$
and $\gamma=\tfrac{1}{\beta}$ and corresponding optimal rate bound $\sqrt{\tfrac{\beta/\sigma-1}{\beta/\sigma+1}}$.
\label{thm:rate_general}
\end{thm}
\begin{proof}
It follows immediately from Theorem~\ref{thm:refl_res_contr}, Lemma~\ref{lem:avg_contr}, and Proposition~\ref{prp:A_sm_l_optparam} by noting that $\alpha=1$ minimizes \eqref{eq:rate_factor_general}.
\end{proof}
In the following section, we will see that there exists a problem from the considered class that converges exactly with the provided rate.

\subsection{Tightness}

We consider a problem where $A$ is a rotation operator, i.e., the it is given by
\begin{align}
A=d\begin{bmatrix}\cos{\psi} & -\sin{\psi}\\
\sin{\psi} & \cos{\psi}
\end{bmatrix}
\label{eq:A_def_lip}
\end{align}
where $0\leq \psi<\tfrac{\pi}{2}$ and $d\in(0,\infty)$. First, we show that $A$ is strongly monotone and Lipschitz continuous.
\begin{prop}
The operator $A$ in \eqref{eq:A_def_lip} is
$d\cos{\psi}$-strongly monotone and $d$-Lipschitz continuous.
\label{prp:A_str_mono_and_Lipschitz}
\end{prop}
\begin{proof}
We first show that $A$ is $d\cos{\psi}$-strongly monotone. Since $A$ is linear, we have
\begin{align*}
\langle Av,v\rangle&=d\langle (\cos{\psi}v_1-\sin{\psi}v_2,\sin{\psi}v_1+\cos{\psi}v_2),(v_1,v_2)\rangle\\
&=d\cos{\psi}(v_1^2+v_2^2)=d\cos{\psi}\|v\|^2.
\end{align*}
That is, $A$ is $d\cos{\psi}$-strongly monotone. Since $A$ is a scaled (with $d$) rotation operator, its largest eigenvalue is $d$, and hence $A$ is $d$-Lipschitz. This concludes the proof.
\end{proof}

We need an explicit form of the reflected resolvent of $A$ to show that the rate is tight. To state it, we define the following alternative arctan definition that is valid when $\tan{\xi}=\tfrac{x}{y}$ and $x\geq 0$:
\begin{align}
\arctan_2\left(\tfrac{x}{y}\right)=\begin{cases}
\arctan(\tfrac{x}{y}) & {\hbox{if }} x\geq 0, y>0\\
\arctan(\tfrac{x}{y})+\pi& {\hbox{if }} x\geq 0, y<0\\
\tfrac{\pi}{2}& x \geq 0, y=0
\end{cases}
\label{eq:atan2}
\end{align}
This arctan is defined for nonnegative numerators $x$ only, and outputs an angle in the interval $[0,\pi]$.

Next, we provide the expression for the reflected resolvent. To simplify its notation, we let $\sigma$ denote the strong convexity modulus and $\beta$ the Lipschitz constant of $A$, i.e.,
\begin{align}
\sigma &= d\cos{\psi}, & \beta&=d.
\label{eq:s_b_A_sm_lip}
\end{align}
The following result is proven in Appendix~\ref{app:pf_A_sm_l}.
\begin{prop}
The reflected resolvent of $\gamma A$, with $A$ in \eqref{eq:A_def_lip} and $\gamma\in(0,\infty)$, is
\begin{align*}
  R_{\gamma A} = \sqrt{1-\tfrac{4\gamma\sigma}{1+2\gamma\sigma+(\gamma\beta)^2}}\begin{bmatrix}
\cos{\xi}&\sin{\xi}\\
-\sin{\xi} & \cos{\xi}
\end{bmatrix}
\end{align*}
where $\sigma$ and $\beta$ are defined in \eqref{eq:s_b_A_sm_lip},
and $\xi$ satisfies $\xi=\arctan_2\left(\tfrac{2\gamma\sqrt{\beta^2-\sigma^2}}{1-(\gamma\beta)^2}\right)$ with $\arctan_2$ defined in \eqref{eq:atan2}.
\label{prp:A_sm_lip_refl_contr}
\end{prop}

That is, the reflected resolvent is first a rotation then a contraction. The contraction factor is exactly the upper bound on the contraction factor in Theorem~\ref{thm:rate_general}. Therefore, the $A$ in \eqref{eq:A_def_lip} can be used to show tightness of the results in Theorem~\ref{thm:rate_general}. To do so, we need another operator $B$ that cancels the rotation introduced by $A$. For $\alpha\in(0,1]$, we will need $R_{\gamma A}R_{\gamma B}=\sqrt{1-\tfrac{4\gamma\sigma}{1+2\gamma\sigma+(\gamma\beta)^2}}I$ and for $\alpha>1$, we will need $R_{\gamma A}R_{\gamma B}=-\sqrt{1-\tfrac{4\gamma\sigma}{1+2\gamma\sigma+(\gamma\beta)^2}}I$. This is clearly achieved if $R_{\gamma B}$ is another rotation operator. Using the following straightforward consequence of Minty's theorem (see \cite{minty1962}) we conclude that any rotation operator (since they are nonexpansive) is the reflected resolvent of a maximally monotone operator.
\begin{prop}
An operator $R~:~\hilbert\to\hilbert$ is nonexpansive if and only if it is the reflected resolvent of a maximally monotone operator.
\end{prop}
\begin{proof}
It follows immediately from \cite[Corollary~23.8]{bauschkeCVXanal} and \cite[Proposition~4.2]{bauschkeCVXanal}.
\end{proof}

With this in mind, we can state the tightness claim.
\begin{prop}
Let $\gamma\in(0,\infty)$, $\delta=\sqrt{1-\tfrac{4\gamma\sigma}{1+2\gamma\sigma+(\gamma\beta)^2}}$, and $\xi$ be defined as in Proposition~\ref{prp:A_sm_lip_refl_contr}. Suppose that $A$ is as in \eqref{eq:A_def_lip} and $B$ is maximally monotone and satisfies either of the following:
\begin{enumerate}[(i)]
\item if $\alpha\in(0,1]$: $B=B_1$ with $R_{\gamma B_1}=\left[\begin{smallmatrix}\cos{\xi}&-\sin{\xi}\\
\sin{\xi}&\cos{\xi}\end{smallmatrix}\right]$,
\item $\alpha\in(1,\tfrac{2}{1+\delta})$: $B=B_2$ with $R_{\gamma B_2}=\left[\begin{smallmatrix}\cos{(\pi-\xi)}&\sin{(\pi-\xi)}\\
-\sin{(\pi-\xi)}&\cos{(\pi-\xi)}\end{smallmatrix}\right]$.
\end{enumerate}
Then the $z^k$ sequence for solving $0\in \gamma Ax+\gamma Bx$ using \eqref{eq:DR_alg} converges exactly with the rate $|1-\alpha|+\alpha\delta$.
\label{prp:A_sm_lip_tight}
\end{prop}
\begin{proof}
Case {\it{(i)}}: Using the reflected resolvent $R_{\gamma A}$ in Proposition~\ref{prp:A_sm_lip_refl_contr} and that $\alpha\in(0,1]$, we conclude that
\begin{align*}
z^{k+1} &= (1-\alpha)z^k+\alpha R_{\gamma A}R_{\gamma B}z^k\\
&=(1-\alpha)z^k+\alpha \delta \begin{bmatrix}\cos{\xi}&\sin{\xi}\\
-\sin{\xi}&\cos{\xi}\end{bmatrix}\begin{bmatrix}\cos{\xi}&-\sin{\xi}\\
\sin{\xi}&\cos{\xi}\end{bmatrix}z^k\\
&=(1-\alpha)z^k+\alpha \delta z^k\\
&=|1-\alpha|z^k+\alpha \delta z^k
\end{align*}
Case {\it{(ii)}}: Using the reflected resolvent $R_{\gamma A}$ in Proposition~\ref{prp:A_sm_lip_refl_contr} and that $\alpha\geq 1$, we conclude that
\begin{align*}
z^{k+1} &= (1-\alpha)z^k+\alpha R_{\gamma A}R_{\gamma B}z^k\\
&=(1-\alpha)z^k+\alpha \delta \begin{bmatrix}\cos{\xi}&\sin{\xi}\\
-\sin{\xi}&\cos{\xi}\end{bmatrix}\begin{bmatrix}\cos{(\pi-\xi)}&\sin{(\pi-\xi)}\\
-\sin{(\pi-\xi)}&\cos{(\pi-\xi)}\end{bmatrix}z^k\\
&=(1-\alpha)z^k-\alpha \delta z^k\\
&=-(|1-\alpha|z^k+\alpha \delta )z^k.
\end{align*}
In both cases, the convergence rate is exactly $|1-\alpha|+\alpha\sqrt{1-\tfrac{4\gamma\sigma}{1+2\gamma\sigma+(\gamma\beta)^2}}$. This completes the proof.
\end{proof}

\begin{rem}
It can be shown that the maximally monotone operator $B_1$ that gives $R_{\gamma B_1}$ satisfies
$B_1=\tfrac{1}{\gamma(1+\cos{\xi})}\left[\begin{smallmatrix}0&-\sin{\xi}\\
\sin{\xi}&0\end{smallmatrix}\right]$ if $\xi\in[0,\pi)$ and $B_1=\partial\iota_0$ (that is, $B_1$ is the subdifferential operator of the indicator function $\iota_0$ of the origin) if $\xi=\pi$. Similarly, the maximally monotone operator $B_2$ that gives $R_{\gamma B_2}$ satisfies $B_2=\tfrac{1}{\gamma(1-\cos{\xi})}\left[\begin{smallmatrix}0&-\sin{\xi}\\
\sin{\xi}&0\end{smallmatrix}\right]$ if $\xi\in(0,\pi]$ and $B_2=0$ if $\xi=0$. 
\end{rem}

We have shown that the rate provided
in Theorem~\ref{thm:rate_general} is tight for all feasible $\alpha$ and $\gamma$.



\subsection{Comparison to other bounds}

In Figure~\ref{fig:rate_comp_general}, we have compared the
linear convergence rate result in 
Theorem~\ref{thm:rate_general} to the convergence rate result in
\cite{LionsMercier1979}. The comparison is made with optimal
$\gamma$-parameters for both bounds. The result in
\cite{LionsMercier1979} is provided in the standard Douglas-Rachford setting,
i.e., with $\alpha=1/2$. By instead letting
$\alpha=1$, this rate can be improved, see
\cite{Davis_Yin_2014} (which shows an improved rate in the composite convex optimization case, but the same rate can be shown to hold also for monotone inclusion problems). Also this
improved rate is added to the comparison in
Figure~\ref{fig:rate_comp_general}. We see that
both rates that follow from \cite{LionsMercier1979}
 are suboptimal and worse than the rate bound in
Theorem~\ref{thm:rate_general}.

\begin{figure}
\centering
\begin{tikzpicture}

    \begin{axis}
      [
      width=10cm,
      height=7cm,
      xtick={1,2,...,10},
      ytick={0,0.2,...,1},
      ymax = 1,
      ymin = 0,
      yscale = 0.7,
      ylabel = {\qquad\quad Convergence factor},
      xlabel = {Ratio $\beta/\sigma$},
      xmax = 10,
      xmin = 1,
      legend cell align=left,
      legend pos=south east,
      ]
      \addplot[blue,domain=1:10,samples=500]{((x-1)/(1+x))^(0.5)};
      \addplot[green,dashed,domain=1:10,samples=500]{(1-1/x)^0.5};
      \addplot[red,dash dot,domain=1:10,samples=500]{(1-1/(2*x))^0.5};
      \legend{Theorem~\ref{thm:rate_general},Improvement to \cite{LionsMercier1979},\cite{LionsMercier1979}};
      \addplot[black,domain=1:10,samples=2]{1};
\end{axis}
\end{tikzpicture}
\caption{Convergence rate comparison for general monotone 
inclusion problems where one operator is strongly monotone and
Lipschitz continuous. We compare Theorem~\ref{thm:rate_general} to
  \cite{LionsMercier1979}, and an improvement to
  \cite{LionsMercier1979} which holds when $\alpha=1$.}
\label{fig:rate_comp_general}

\end{figure}
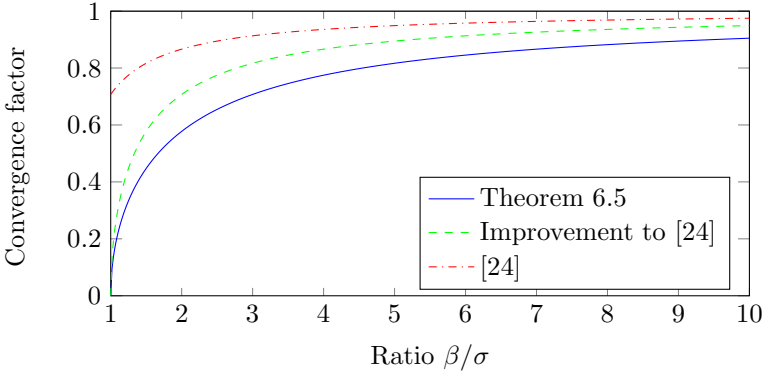

\section{$A$ strongly monotone and cocoercive}

\label{sec:A_sm_c}

In this section, we consider the case where $A$ is strongly monotone and cocoercive. That is, we assume the following.
\begin{ass}
Suppose that:
\begin{enumerate}[(i)]
\item The operators $A~:~\hilbert\to\hilbert$ and $B~:~\hilbert\to 2^{\hilbert}$ are maximally monotone.
\item $A$ is $\sigma$-strongly monotone and $\tfrac{1}{\beta}$-cocoercive.
\end{enumerate}
\label{ass:str_coco_same_operator}
\end{ass}

The linear convergence result for Douglas-Rachford splitting will follow from the contraction factor of the reflected resolvent of $A$. The contraction factor is provided in the following theorem, which is proven in Appendix~\ref{app:pf_A_sm_c}.
\begin{thm}
Suppose that $A~:~\hilbert\to \hilbert$ is a $\sigma$-strongly monotone and $\tfrac{1}{\beta}$-cocoercive operator. Then
its reflected resolvent $R_{A} = 2J_A-\id$ is contractive with factor
$\sqrt{1-\tfrac{4\sigma}{1+2\sigma+\sigma\beta}}$.
\label{thm:A_sm_c_contraction}
\end{thm}

When considering the reflected resolvent of $\gamma A$ where $\gamma\in(0,\infty)$, the $\gamma$-parameter can be chosen to optimize the contraction factor of $R_{\gamma A}$. The operator $\gamma A$ is $\gamma\sigma$-strongly monotone and $\tfrac{1}{\gamma\beta}$-cocoercive, so the optimal $\gamma>0$ minimizes $h(\gamma):=1-\tfrac{4\gamma\sigma}{1+2\gamma\sigma+\gamma^2\sigma\beta}$. The gradient of $h$ satisfies $\nabla h(\gamma) = \tfrac{4\sigma(\beta\sigma\gamma^2-1)}{(\beta\sigma\gamma^2+2\sigma\gamma+1)^2}$, so the extreme points of $h$ are given by $\gamma=\pm\tfrac{1}{\sqrt{\beta\sigma}}$. Since $\gamma>0$ and the gradient is negative for $\gamma\in(0,\tfrac{1}{\sqrt{\beta\sigma}})$ and positive for $\gamma>\tfrac{1}{\sqrt{\beta\sigma}}$, the parameter $\gamma=\tfrac{1}{\sqrt{\beta\sigma}}$ minimizes the contraction factor. The corresponding contraction factor is
\begin{align*}
\sqrt{1-\tfrac{4\gamma\sigma}{1+2\gamma\sigma+\gamma^2\sigma\beta}} = \sqrt{1-\tfrac{2\sqrt{\sigma/\beta}}{1+\sqrt{\sigma}{\beta}}}=\sqrt{\tfrac{1-\sqrt{\sigma/\beta}}{1+\sqrt{\sigma/\beta}}}=\sqrt{\tfrac{\sqrt{\beta/\sigma}-1}{\sqrt{\beta/\sigma}+1}}.
\end{align*}
This is summarized in the following proposition.

\begin{prop}
  The parameter $\gamma\in(0,\infty)$ that optimizes the contraction factor for $R_{\gamma A}$ is $\gamma=\tfrac{1}{\sqrt{\beta\sigma}}$. The corresponding contraction factor is $\sqrt{\tfrac{\sqrt{\beta/\sigma}-1}{\sqrt{\beta/\sigma}+1}}$.
\label{prp:A_sm_c_optparam}
\end{prop}

Now we are ready to state the linear convergence rate result for the Douglas-Rachford algorithm.

\begin{thm}
Suppose that Assumption~\ref{ass:str_coco_same_operator} holds and that the
Douglas-Rachford algorithm \eqref{eq:DR_alg} is applied to solve $0\in
\gamma Ax+\gamma Bx$. Let $\delta = \sqrt{1-\tfrac{4\gamma\sigma}{1+2\gamma\sigma+\gamma^2\sigma\beta}}$, then
the algorithm converges at least with rate factor
\begin{align}
|1-\alpha|+\alpha\delta
\label{eq:A_sm_c_rate}
\end{align}
for all $\alpha\in(0,\tfrac{2}{1+\delta})$. Optimizing this bound w.r.t. $\alpha$ and $\gamma$ gives $\alpha=1$
and $\gamma=\tfrac{1}{\sqrt{\beta\sigma}}$ and corresponding optimal rate bound $\sqrt{\tfrac{\sqrt{\beta/\sigma}-1}{\sqrt{\beta/\sigma}+1}}$.
\label{thm:A_sm_c_rate}
\end{thm}
\begin{proof}
It follows immediately from Theorem~\ref{thm:A_sm_c_contraction}, Lemma~\ref{lem:avg_contr}, and Proposition~\ref{prp:A_sm_c_optparam} by noting that $\alpha=1$ minimizes \eqref{eq:A_sm_c_rate}.
\end{proof}

\subsection{Tightness}

In this section, we provide a two-dimensional example that shows that the provided bounds are tight. We let $A$ be the resolvent of a scaled rotation operator to achieve this. Let $C$ be that scaled rotation operator, i.e.,
\begin{align}
C=c\begin{bmatrix}
\cos{\psi} & -\sin{\psi}\\
\sin{\psi} & \cos{\psi}
\end{bmatrix}
\label{eq:C_rotmat}
\end{align}
with $c\in (1,\infty)$ and $\psi\in[0,\tfrac{\pi}{2})$. We will let $A$ satisfy $A=dJ_{C}$ for some $d\in(0,\infty)$. That is
\begin{align}
A=d(C+I)^{-1}=\tfrac{d}{(1+c\cos{\psi})^2+c^2\sin^2{\psi}}
\begin{bmatrix}
c\cos{\psi}+1 & c\sin{\psi}\\
-c\sin{\psi} & c\cos{\psi}+1
\end{bmatrix}.
\label{eq:A_mat_coco}
\end{align}
In the following proposition, we state the strong monotonicity and cocoercivity properties of $A$.
\begin{prop}
The operator $A$ in \eqref{eq:A_mat_coco} is $\tfrac{1+c\cos{\psi}}{d}$-cocoercive and strongly monotone with modulus $\tfrac{d(1+c\cos{\psi})}{1+2c\cos{\psi}+c^2}$.
\end{prop}
\begin{proof}
The matrix $C$ in \eqref{eq:C_rotmat} is $c\cos{\psi}$-strongly monotone (see Proposition~\ref{prp:A_str_mono_and_Lipschitz}), so $J_C$ is $(1+c\cos{\psi})$-cocoercive (see \cite[Definition~4.4]{bauschkeCVXanal}) and the operator $A=d(I+C)^{-1}$ is $\tfrac{1+c\cos{\psi}}{d}$-cocoercive. Further, since $C$ is monotone and $c$-Lipschitz continuous (see Proposition~\ref{prp:A_str_mono_and_Lipschitz}), the following holds (see Proposition~\ref{prp:res_str_mono_general}):
\begin{align}
2\langle J_{C}x-J_Cy,x-y\rangle\geq \|x-y\|^2+(1-c^2)\|J_Cx-J_Cy\|^2.
\label{eq:C_lipschitz}
\end{align}
Since $J_C$ is $(1+c\cos{\psi})$-cocoercive, we have
\begin{align}
\langle J_Cx-J_Cy,x-y\rangle\geq (1+c\cos{\psi})\|J_Cx-J_Cy\|^2.
\label{eq:C_sm}
\end{align}
We add \eqref{eq:C_sm} multiplied by $-\tfrac{1-c^2}{1+c\cos{\psi}}$ (which is positive since $c\in(1,\infty)$) to \eqref{eq:C_lipschitz} to get 
\begin{align*}
(2-\tfrac{1-c^2}{1+c\cos{\psi}})\langle J_{C}x-J_Cy,x-y\rangle\geq \|x-y\|^2.
\end{align*}
That is, $J_C$ is $\sigma$-strongly monotone with
\begin{align*}
\frac{1}{2-\tfrac{1-c^2}{1+c\cos{\psi}}}=\frac{1+c\cos{\psi}}{2+2c\cos{\psi}-1+c^2}=\frac{1+c\cos{\psi}}{1+2c\cos{\psi}+c^2},
\end{align*}
so $A$ is strongly monotone with parameter $d\tfrac{1+c\cos{\psi}}{1+2c\cos{\psi}+c^2}$. This concludes the proof. 
\end{proof}

This shows that the assumptions needed for the linear convergence rate result in Theorem~\ref{thm:A_sm_c_rate} hold. To prove the tightness claim, we need an expression for the reflected resolvent of $A$. This is easier expressed in the strong convexity modulus, which we define as $\sigma$ and the inverse cocoercivity constant, which we define as $\beta$, i.e.,:
\begin{align}
\sigma&=\tfrac{d(1+c\cos{\psi})}{1+2c\cos{\psi}+c^2},&\beta&=\tfrac{d}{1+c\cos{\psi}}.
\label{eq:sig_sm_beta_coco}
\end{align}
The following results is proven in Appendix~\ref{app:pf_A_sm_c}.

\begin{prop}
The reflected resolvent $R_{\gamma A}$ of $\gamma A$, where $A$ is defined in \eqref{eq:A_mat_coco} and $\gamma\in(0,\infty)$, is given by
\begin{align*}
R_{\gamma A}=\sqrt{1-\frac{4\gamma\sigma}{1+2\gamma\sigma+\gamma^2\sigma\beta}}\begin{bmatrix}
\cos{\xi} & -\sin{\xi}\\
\sin{\xi} & \cos{\xi}
\end{bmatrix}
\end{align*}
where $\sigma$ and $\beta$ are defined in \eqref{eq:sig_sm_beta_coco}, and $\xi$ satisfies $\xi =\arctan_2\left(\tfrac{2\gamma\sqrt{\sigma(\beta-\sigma)}}{1-\sigma\beta\gamma^2}\right)$ with $\arctan_2$ defined in \eqref{eq:atan2}. 
\label{prp:A_sm_c_refl_contr}
\end{prop}

Based on this reflected resolvent, we can show that the rate bound in Theorem~\ref{thm:A_sm_c_rate} is indeed tight. The proof of the following result is the same as the proof to Proposition~\ref{prp:A_sm_lip_tight}.
\begin{prop}
Let $\gamma\in(0,\infty)$, $\delta=\sqrt{1-\tfrac{4\gamma\sigma}{1+2\gamma\sigma+\gamma^2\sigma\beta}}$, and let $\xi$ be as defined in Proposition~\ref{prp:A_sm_c_refl_contr}. Suppose that $A$ is as in \eqref{eq:A_mat_coco} and $B$ is maximally monotone and satisfies either of the following:
\begin{enumerate}[(i)]
\item if $\alpha\in(0,1]$: $B=B_1$ with $R_{\gamma B_1}=\left[\begin{smallmatrix}\cos{\xi}&\sin{\xi}\\
-\sin{\xi}&\cos{\xi}\end{smallmatrix}\right]$,
\item $\alpha\in(1,\tfrac{2}{1+\delta})$: $B=B_2$ with $R_{\gamma B_2}=\left[\begin{smallmatrix}\cos{(\pi-\xi)}&-\sin{(\pi-\xi)}\\
\sin{(\pi-\xi)}&\cos{(\pi-\xi)}\end{smallmatrix}\right]$.
\end{enumerate}
Then the $z^k$ sequence for solving $0\in \gamma Ax+\gamma Bx$ using \eqref{eq:DR_alg} converges exactly with the rate $|1-\alpha|+\alpha\delta$.
\end{prop}

So, we have shown that the rate in Theorem~\ref{thm:A_sm_c_rate} is tight for all feasible algorithm parameters $\alpha$ and $\gamma$.

\subsection{Comparison to other bounds}


\begin{figure}
\centering
\begin{tikzpicture}

    \begin{axis}
      [
      width=10cm,
      height=7cm,
      xtick={1,2,...,10},
      ytick={0,0.2,...,1},
      ymax = 1,
      ymin = 0,
      yscale = 0.7,
      ylabel = {\qquad\quad Convergence factor},
      xlabel = {Ratio $\beta/\sigma$},
      xmax = 10,
      xmin = 1,
      legend cell align=left,
      legend pos=south east,
      ]
      \addplot[blue,domain=1:10,samples=500]{(x^0.5-1)/(1+x^0.5)};
      \addplot[green,dashed,domain=1:10,samples=500]{((x^0.5-1)/(1+x^0.5))^(0.5)};
      \addplot[red,dash dot,domain=1:10,samples=500]{((x-1)/(1+x))^(0.5)};
      \legend{\cite{gisBoydTAC2014metric_select},Theorem~\ref{thm:A_sm_c_rate},Theorem~\ref{thm:rate_general}};
      \addplot[black,domain=1:10,samples=2]{1};
\end{axis}
\end{tikzpicture}
\caption{Convergence rate comparison between Theorem~\ref{thm:rate_general}, Theorem~\ref{thm:A_sm_c_rate}, and \cite{gisBoydTAC2014metric_select}.
In all, one operator has both regularity properties. It is strongly monotone in all examples and Lipschitz in Theorem~\ref{thm:rate_general}, cocoercive in Theorem~\ref{thm:A_sm_c_rate} (which is stronger than Lipschitz), and a cocoercive subdifferential operator in \cite{gisBoydTAC2014metric_select} (which is the strongest property). The worst-case rate improves when the class of problems becomes more restricted.} 
\label{fig:rate_comp_gen_vs_subdiff}

\end{figure}
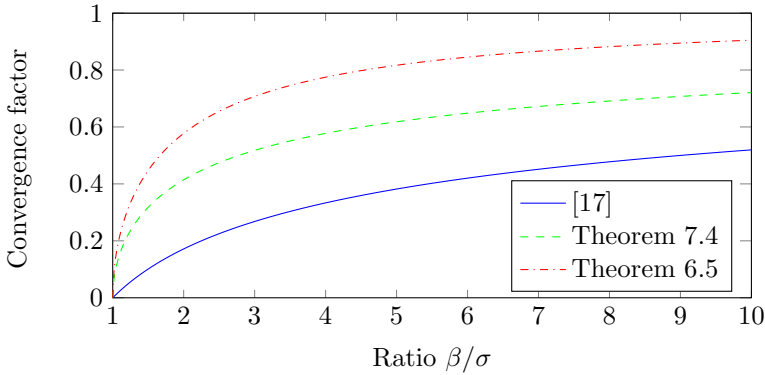

We have shown tight convergence rate estimates for Douglas-Rachford splitting when the monotone operator $A$ is cocoercive and strongly monotone (Theorem~\ref{thm:A_sm_c_rate}). In Section~\ref{sec:A_sm_l}, we showed tight estimates when $A$ is Lipschitz and strongly monotone (Theorem~\ref{thm:rate_general}). In \cite{gisBoydTAC2014metric_select}, tight convergence rate estimates are proven for the case when $A$ and $B$ are subdifferential operators of proper closed and convex functions and $A$ is strongly monotone and Lipschitz continuous (which in this case is equivalent to cocoercive). The class of problems considered in \cite{gisBoydTAC2014metric_select} is a subclass of the problems considered in this section, which in turn is a subclass of the problems considered in Section~\ref{sec:A_sm_l}. The optimal rates for these classes of problems are shown in Figure~\ref{fig:rate_comp_gen_vs_subdiff}. By restricting the problem classes, the rate bounds get tighter. This is in contrast to the case in Section~\ref{sec:split_prop}, where a convex optimization problem achieved the worst case estimate.

\section{Conclusions}

We have shown linear convergence rate bounds for Douglas-Rachford
splitting for monotone inclusion problems with three different sets of assumptions. One setting was the one used by Lions and Mercier \cite{LionsMercier1979}, for which we provided a tighter bound. We also stated linear convergence rate bounds under two other assumptions, for which no other linear rate bounds were previously available. In addition, we have shown that all our rate bounds are tight for, in two cases all feasible algorithm parameters, and in the remaining case many algorithm parameters. 

\subsection*{Acknowledgment}
The author would like to thank Heinz Bauschke for suggesting the term \emph{negatively averaged operators}.

\begin{appendix}

\section{Proofs to Lemmas in Section~\ref{sec:operator_prop_rel}}

\label{app:lemma_proofs}

\subsection{Proof to Lemma~\ref{lem:coco_vs_lipschitz}}

From the definition of cocoercivity, Definition~\ref{def:cocoercive}, it follows directly that
$\beta\id+T$ is $\tfrac{1}{2\beta}$-cocoercive if and only if
$\tfrac{1}{2\beta}(\beta\id+T)$ is 1-cocoercive. This, in turn is equivalent to that
$2\tfrac{1}{2\beta}(\beta\id+T)-\id=\tfrac{1}{\beta}T$ is
nonexpansive \cite[Proposition~4.2 and
Definition~4.4]{bauschkeCVXanal}. Finally, from the definition of
Lipschitz continuity, Definition~\ref{def:Lipschitz}, it follows
directly that $\tfrac{1}{\beta} T$ is nonexpansive if and only if
$T$ is $\beta$-Lipschitz continuous. This concludes the proof.

\subsection{Proof to Lemma~\ref{lem:coco_averaged}}
Let $T_1 = R-\tfrac{\beta}{2}\id$. Then Lemma~\ref{lem:coco_vs_lipschitz} states that
$\tfrac{1}{\beta}$-cocoercivity of $T_1+\tfrac{\beta}{2}\id=R$ is equivalent to 
$\tfrac{\beta}{2}$-Lipschitz continuity of $T_1=R-\tfrac{\beta}{2}\id$.
By definition of Lipschitz continuity, this is equivalent to that $T_1 =
\tfrac{\beta}{2}T_2$ for some nonexpansive operator $T_2$. Therefore
$T=R+(1-\beta)\id =
T_1+(1-\tfrac{\beta}{2})\id=\tfrac{\beta}{2}T_2+(1-\tfrac{\beta}{2})\id$.
Since $\beta\in(0,1)$, this is equivalent to that $T$ is
$\tfrac{\beta}{2}$-averaged. This concludes the proof.

\subsection{Proof to Lemma~\ref{lem:avg_contr}}

Let $x$ and $y$ be any points in $\hilbert$. Then
\begin{align*}
\|Rx-Ry\|&=\|(1-\alpha)x+\alpha Tx-(1-\alpha)y-\alpha Ty\|\\
&\leq |1-\alpha|\|x-y\|+|\alpha|\|Tx-Ty\|\\
&\leq |1-\alpha|\|x-y\|+|\alpha|\delta\|x-y\|\\
&= (|1-\alpha|+|\alpha|\delta)\|x-y\|.
\end{align*}
So $R$ is $(|1-\alpha|\|x-y\|+|\alpha|\delta)$-Lipschitz continuous. The Lipschitz constant is less than 1 if $\alpha\in(0,\tfrac{2}{1+\delta})$. For such $\alpha$, $R$ is contractive. Since $\alpha>0$, the contraction factor is $(|1-\alpha|+\alpha\delta)$. This concludes the proof.

\section{Proofs to results in Section~\ref{sec:split_prop}}
\label{app:pf_split_prop}

\subsection{Proof to Proposition~\ref{prp:res_prop_lipschitz_subdiff}}

Since $B$ is $\tfrac{1}{\beta}$-cocoercive, it satisfies
\begin{align*}
\langle Bu-Bv,u-v\rangle\geq\tfrac{1}{\beta}\|Bu-Bv\|^2.
\end{align*}
Adding $\|u-v\|^2$ to both sides gives
\begin{align*}
\langle (B+\id)u-(B+\id)v,u-v\rangle&\geq\|u-v\|^2\\
&\quad+\tfrac{1}{\beta}\|(B+\id)u-(B+\id)v-(u-v)\|^2.
\end{align*}
Letting $x= (B+\id)u$ and $y= (B+\id)v$ implies that $u= J_{B}x$ and $v= J_By$. Therefore, we get the equivalent expression
\begin{align*}
\langle x-y,J_Bx-J_By\rangle\geq\|J_Bx-J_By\|^2+\tfrac{1}{\beta}\|x-y-(J_Bx-J_By)\|^2.
\end{align*}
Expansion of the second square gives
\begin{align*}
&\beta\langle x-y,J_Bx-J_By\rangle\\
&\geq\beta\|J_Bx-J_By\|^2+\|x-y\|^2-2\langle x-y,J_Bx-J_By\rangle+\|J_Bx-J_By\|^2,
\end{align*}
or equivalently
\begin{align*}
&\quad(\beta+2)\langle x-y,J_Bx-J_By\rangle\geq\|x-y\|^2+(\beta+1)\|J_Bx-J_By\|^2\\
\Leftrightarrow&\quad\tfrac{\beta+2}{\beta+1}\langle x-y,J_Bx-J_By\rangle\geq\tfrac{1}{\beta+1}\|x-y\|^2+\|J_Bx-J_By\|^2\\
\Leftrightarrow&\quad(2(1-\tfrac{\beta}{2(1+\beta)}))\langle x-y,J_Bx-J_By\rangle\\
&\quad\qquad\qquad\qquad\qquad\qquad\geq(1-2\tfrac{\beta}{2(\beta+1)})\|x-y\|^2+\|J_Bx-J_By\|^2.
\end{align*}
This is, by \cite[Proposition~4.25]{bauschkeCVXanal}, equivalent to that $J_B$ is $\tfrac{\beta}{2(\beta+1)}$-averaged. This concludes the proof.

\subsection{Proof to Theorem~\ref{thm:rate_subdiff}}
\label{app:thm_split_prop}

Since $R_{\gamma A}R_{\gamma B}$ is
$\tfrac{\tfrac{1}{\gamma\sigma}+\gamma\beta}{\tfrac{1}{\gamma\sigma}+\gamma\beta+1}$-negatively averaged,
see Proposition~\ref{prp:comp_refl_res_anti_averaged_subdiff},
the Douglas-Rachford iteration is defined by an $\alpha$-averaged
$\tfrac{\tfrac{1}{\gamma\sigma}+\gamma\beta}{\tfrac{1}{\gamma\sigma}+\gamma\beta+1}$-negatively
averaged
operator. The rate in \eqref{eq:DR_rate_bound_AB_prop_subdiff} follows
directly from Proposition~\ref{prp:avg_anti_avg_contr}. The optimal
parameters follow from
Proposition~\ref{prp:avg_antiavg_optparam_subdiff}. It shows that the
rate factor is increasing in 
$\tfrac{\tfrac{1}{\gamma\sigma}+\gamma\beta}{\tfrac{1}{\gamma\sigma}+\gamma\beta+1}$,
which in turn is increasing in $\tfrac{1}{\gamma\sigma}+\gamma\beta$.
Therefore this should be minimized to optimize the rate. The optimal
$\gamma=\tfrac{1}{\sqrt{\beta\sigma}}$ gives negative averagedness factor
$\tfrac{\tfrac{1}{\gamma\sigma}+\gamma\beta}{\tfrac{1}{\gamma\sigma}+\gamma\beta+1}=\tfrac{2\sqrt{\beta/\sigma}}{1+2\sqrt{\beta/\sigma}}$.
Proposition~\ref{prp:avg_antiavg_optparam_subdiff} further gives that
the optimal averagedness factor is
\begin{align*}
\alpha = \tfrac{1}{2-\tfrac{2\sqrt{\beta/\sigma}}{1+2\sqrt{\beta/\sigma}}}=\tfrac{1+2\sqrt{\beta/\sigma}}{2+2\sqrt{\beta/\sigma}}=\tfrac{1/2+\sqrt{\beta/\sigma}}{1+\sqrt{\beta/\sigma}}
\end{align*}
and that the optimal bound on the contraction factor is
\begin{align*}
\tfrac{\tfrac{2\sqrt{\beta/\sigma}}{2\sqrt{\beta/\sigma}+1}}{2-\tfrac{2\sqrt{\beta/\sigma}}{2\sqrt{\beta/\sigma}+1}}=\tfrac{2\sqrt{\beta/\sigma}}{2+2\sqrt{\beta/\sigma}}=\tfrac{\sqrt{\beta/\sigma}}{1+\sqrt{\beta/\sigma}}.
\end{align*}
This concludes the proof.

\subsection{Proof to Proposition~\ref{prp:tight_split_prop}}

The proximal and reflected proximal operators of $f$ are trivially given by
\begin{align}
\label{eq:f_prox}{\rm{prox}}_{\gamma f}(y)&=(\tfrac{1}{1+\gamma\beta}y_1,y_2),&R_{\gamma f}(y)&=(\tfrac{1-\gamma\beta}{1+\gamma\beta}y_1,y_2).
\end{align}
Linearity of the proximal operator and Moreau's decomposition \cite[Theorem~14.3]{bauschkeCVXanal} imply that the reflected resolvent of $f^*$ is given by
\begin{align*}
R_{\gamma f^*} &= 2{\rm{prox}}_{\gamma
  f^*}-\id=2(\id-\gamma{\rm{prox}}_{\gamma^{-1}f}\circ(\gamma^{-1}\id))-\id=-(2{\rm{prox}}_{\gamma^{-1}f}-\id) \\
&=-R_{\gamma^{-1}f}.
\end{align*}
This gives the following Douglas-Rachford iteration:
\begin{align*}
z^{k+1}&=((1-\alpha)\id+\alpha R_{\gamma f}R_{\gamma f^*})z^k\\
&=(1-\alpha)z^k-\alpha R_{\gamma f}R_{\gamma^{-1}f}z^k\\
&=(1-\alpha)z^k-\alpha \left(\tfrac{(1-\gamma\beta)(1-\gamma^{-1}\beta)}{(1+\gamma\beta)(1+\gamma\beta)}z_1^k,z_2^k\right).
\end{align*}
Since we start at a point $z^0=(0,z_2^0)$, we will get $z_1^k=0$ for all $k\geq 1$, and the Douglas-Rachford iteration becomes
\begin{align*}
z^{k+1} = \left(1-2\alpha\right)z^k
\end{align*}
with contraction factor given by $|1-2\alpha|$. 

When $\alpha\in[c,1)$, the absolute value term in \eqref{eq:DR_rate_bound_AB_prop_subdiff} is nonpositive since
\begin{align*}
(1-2\alpha+\alpha\tfrac{\tfrac{1}{\gamma\sigma}+\gamma\beta}{1+\tfrac{1}{\gamma\sigma}+\gamma\beta})&\leq
0 &\Leftrightarrow && \alpha&\geq\tfrac{1}{2-\tfrac{\tfrac{1}{\gamma\sigma}+\gamma\beta}{1+\tfrac{1}{\gamma\sigma}+\gamma\beta}}=\tfrac{1+\tfrac{1}{\gamma\sigma}+\gamma\beta}{2+\tfrac{1}{\gamma\sigma}+\gamma\beta}=c.
\end{align*}
Therefore, for such $\alpha$, the rate in \eqref{eq:DR_rate_bound_AB_prop_subdiff} is $|1-2\alpha|$. This coincides with the rate for the provided example for any $\gamma>0$, and the proof is completed.

\section{Proofs to results in Section~\ref{sec:A_sm_l}}
\label{app:pf_A_sm_l}

\subsection{Proof to Proposition~\ref{prp:res_str_mono_general}}

$\beta$-Lipschitz continuity of $A$ implies that $\beta\id+A$ is
$\tfrac{1}{2\beta}$-cocoercive, see Lemma~\ref{lem:coco_vs_lipschitz}. That is
\begin{align*}
\langle (\beta\id&+A)u-(\beta\id+A)v,u-v\rangle
\geq\tfrac{1}{2\beta}\|(\beta\id+A)u-(\beta\id+A)v\|^2.
\end{align*}
Using $\beta\id = \id+(\beta-1)\id$, this is equivalent to that
\begin{align*}
\langle (\id+A)u-(\id+A)v,u-v\rangle
&\geq\tfrac{1}{2\beta}\|(\id+A)u-(\id+A)v+(\beta-1)(u-v)\|^2\\
&\quad+(1-\beta)\|u-v\|^2.
\end{align*}
Using that $x=(\id+A)u$ if and only if $u= (\id+A)^{-1}x$ and 
$y=(\id+A)v$ if and only if $v= (\id+A)^{-1}y$
(that hold by definition of the inverse and single-valuedness), this is equivalent to
\begin{align*}
\langle x-y,(\id+A)^{-1}x-&(\id+A)^{-1}y\rangle\\
&\geq\tfrac{1}{2\beta}\|x-y+(\beta-1)((\id+A)^{-1}x-(\id+A)^{-1}y)\|^2\\
&\quad+(1-\beta)\|(\id+A)^{-1}x-(\id+A)^{-1}y\|^2.
\end{align*}
Identifying the resolvent $J_A=(\id+A)^{-1}$ and expanding the first square give:
\begin{align*}
\langle &J_Ax-J_Ay,x-y\rangle\\
&\geq\tfrac{1}{2\beta}\|x-y+(\beta-1)(J_Ax-J_Ay)\|^2+(1-\beta)\|J_Ax-J_Ay\|^2\\
&=\tfrac{1}{2\beta}\left(\|x-y\|^2+2(\beta-1)\langle J_Ax-J_Ay,x-y\rangle +(\beta-1)^2\|J_Ax-J_Ay\|^2\right)\\
&\quad+(1-\beta)\|J_Ax-J_Ay\|^2
\end{align*}
By rearranging the terms, we conclude that
\begin{align*}
(1+\tfrac{1-\beta}{\beta})\langle J_Ax-J_Ay,x-y\rangle
&\geq\tfrac{1}{2\beta}\left(\|x-y\|^2+(\beta-1)^2\|J_Ax-J_Ay\|^2\right)\\
&\quad+(1-\beta)\|J_Ax-J_Ay\|^2\\
&=\tfrac{1}{2\beta}\|x-y\|^2+\tfrac{(\beta-1)^2+2\beta(1-\beta)}{2\beta}\|J_Ax-J_Ay\|^2\\
&=\tfrac{1}{2\beta}\|x-y\|^2+\tfrac{1-\beta^2}{2\beta}\|J_Ax-J_Ay\|^2.
\end{align*}
The result follows by multiplying by $2\beta$, since $1+\tfrac{1-\beta}{\beta}=\tfrac{1}{\beta}$.
This concludes the proof.

\subsection{Proof to Theorem~\ref{thm:refl_res_contr}}

We divide the proof into two cases, $\beta\geq 1$ and $\beta\leq 1$.
\subsection*{Case $\beta\geq 1$}
From 
\cite[Proposition~23.11]{bauschkeCVXanal}, we get that
$J_A$ is $(1+\sigma)$-cocoercive, i.e., that
\begin{align}
\langle J_Ax-J_Ay,x-y\rangle\geq (1+\sigma)\|J_Ax-J_Ay\|^2.
\label{eq:res_coco}
\end{align}
Adding $(\beta^2-1)(\geq 0)$ of \eqref{eq:res_coco} to $(1+\sigma)$ of
\eqref{eq:res_prop_A_lipschitz}, we get 
\begin{align*}
(2(1+\sigma)+(\beta^2-1))\langle J_Ax-&J_Ay,x-y\rangle
\geq (1+\sigma)\|x-y\|^2
\end{align*}
or equivalently
\begin{align}
\langle J_Ax-&J_Ay,x-y\rangle
\geq \tfrac{1+\sigma}{1+2\sigma+\beta^2}\|x-y\|^2
\label{eq:add_res_prop_gen}
\end{align}
since the $\|J_Ax-J_Ay\|$ terms cancel. We get
\begin{align}
\nonumber\|R_Ax-R_Ay\|^2 &= \|2J_Ax-2J_Ay-(x-y)\|^2\\
\nonumber&=4\|J_Ax-J_Ay\|^2-4\langle J_Ax-J_Ay,x-y\rangle+\|x-y\|^2\\
\nonumber&\leq 4(\tfrac{1}{1+\sigma}-1)\langle J_Ax-J_Ay,x-y\rangle+\|x-y\|^2\\
\nonumber&= -\tfrac{4\sigma}{1+\sigma}\langle J_Ax-J_Ay,x-y\rangle+\|x-y\|^2\\
\nonumber&\leq
-\tfrac{4\sigma}{1+\sigma}\tfrac{1+\sigma}{1+2\sigma+\beta^2}\|x-y\|^2+\|x-y\|^2\\
\label{eq:refl_res_contr_beta_geq_1}&=(1-\tfrac{4\sigma}{1+2\sigma+\beta^2})\|x-y\|^2
\end{align}
where \eqref{eq:res_coco} and \eqref{eq:add_res_prop_gen} are used in
the inequalities. Thus, the said result holds for $\beta\geq 1$.

\subsection*{Case $\beta\leq 1$}

To prove the result for $\beta\leq 1$, we define the set $\mathcal{R}$
of pairs of points $(x,y)\in\hilbert\times\hilbert$ as follows:
\begin{align}
\mathcal{R} = \left\{(x,y)~|~\langle
  J_Ax-J_Ay,x-y\rangle\geq\tfrac{1+\sigma}{1+2\sigma+\beta^2}\|x-y\|^2\right\}.
\label{eq:R_set_def}
\end{align}
We also define the closure of the remaining pairs of points
$\mathcal{R}_c=\overbar{(\hilbert\times\hilbert)\backslash\mathcal{R}}$, i.e.,
\begin{align}
\mathcal{R}_c = \left\{(x,y)~|~\langle J_Ax-J_Ay,x-y\rangle\leq\tfrac{1+\sigma}{1+2\sigma+\beta^2}\|x-y\|^2\right\}.
\label{eq:Rc_set_def}
\end{align}
Obviously, $\hilbert\times\hilbert\subseteq\mathcal{R}+\mathcal{R}_c$
which implies that the contraction factor of the resolvent is the
worst-case contraction factor for $\mathcal{R}$ and
$\mathcal{R}_c$. We first show the contraction factor for
$\mathcal{R}$. Since \eqref{eq:add_res_prop_gen} is the
definition of the set $\mathcal{R}$ in \eqref{eq:R_set_def}, the
contraction factor for $(x,y)\in\mathcal{R}$ is shown exactly as in
\eqref{eq:refl_res_contr_beta_geq_1}. For $(x,y)\in\mathcal{R}_c$, we
have
\begin{align*}
\|R_Ax-R_Ay\|^2 &= \|2J_Ax-2J_Ay-(x-y)\|^2\\
&=4\|J_Ax-J_Ay\|^2-4\langle J_Ax-J_Ay,x-y\rangle+\|x-y\|^2\\
&\leq \tfrac{4}{1-\beta^2}\left(2\langle
  J_Ax-J_Ay,x-y\rangle-\|x-y\|^2\right)\\
&\quad-4\langle J_Ax-J_Ay,x-y\rangle+\|x-y\|^2\\
&=4\left(\tfrac{2}{1-\beta^2}-1\right)\langle
J_Ax-J_Ay,x-y\rangle+\left(1-\tfrac{4}{1-\beta^2}\right)\|x-y\|^2\\
&\leq\left(1-\tfrac{4}{1-\beta^2}+4\tfrac{1+\beta^2}{1-\beta^2}\tfrac{1+\sigma}{1+2\sigma+\beta^2}\right)\|x-y\|^2\\
&=\left(1-\tfrac{4(1+2\sigma+\beta^2)-4(1+\beta^2)(1+\sigma)}{(1-\beta^2)(1+2\sigma+\beta^2)}\right)\|x-y\|^2\\
&=\left(1-\tfrac{4+8\sigma+4\beta^2-(4+4\beta^2+4\sigma+4\sigma\beta^2)}{(1-\beta^2)(1+2\sigma+\beta^2)}\right)\|x-y\|^2\\
&=\left(1-\tfrac{4\sigma(1-\beta^2)}{(1-\beta^2)(1+2\sigma+\beta^2)}\right)\|x-y\|^2\\
&=\left(1-\tfrac{4\sigma}{1+2\sigma+\beta^2}\right)\|x-y\|^2
\end{align*}
where \eqref{eq:res_prop_A_lipschitz} is used in the first inequality
and the definition of $\mathcal{R}_c$ in \eqref{eq:Rc_set_def} in the second.
That is, the worst case contraction factor is
$\sqrt{1-\tfrac{4\sigma}{1+2\sigma+\beta^2}}$ also for $\beta\leq 1$.

It remains to show that the contraction factor is in the interval $[0,1)$. We
show that the square of the contraction factor is in $[0,1)$. We
have
$1-\tfrac{4\sigma}{1+2\sigma+\beta^2} = \tfrac{1-2\sigma+\beta^2}{1+2\sigma+\beta^2}<1$.
Further, since $\sigma\leq\beta$, we have
$1-2\sigma+\beta^2\geq 1-2\sigma+\sigma^2=(1-\sigma)^2\geq0$. So the
numerator is nonnegative and the denominator is positive, which gives
a nonnegative contraction factor. This concludes the proof.

\subsection{Proof to Proposition~\ref{prp:A_sm_lip_refl_contr}}

First, we compute the resolvent $J_{\gamma A}$. It satisfies
\begin{align*}
J_{\gamma A} &= (I+\gamma A)^{-1}=\begin{bmatrix} 1+\gamma d\cos{\psi} & -\gamma d\sin{\psi}\\
\gamma d\sin{\psi}& 1+\gamma d\cos{\psi}\end{bmatrix}^{-1}\\
&=\frac{1}{1+2\gamma d\cos{\psi}+(\gamma d)^2}\begin{bmatrix} 1+\gamma d\cos{\psi} & \gamma d\sin{\psi}\\
-\gamma d\sin{\psi}& 1+\gamma d\cos{\psi}\end{bmatrix}\\
&=\frac{1}{1+2\gamma \sigma+(\gamma \beta)^2}\begin{bmatrix} 1+\gamma \sigma & \gamma \beta\sin{\psi}\\
-\gamma \beta\sin{\psi}& 1+\gamma \sigma\end{bmatrix}
\end{align*}
The reflected resolvent is
\begin{align*}
R_{\gamma A} &=2J_{\gamma A}-I\\
&=\frac{2}{1+2\gamma \sigma+(\gamma \beta)^2}\begin{bmatrix} 1+\gamma \sigma-\tfrac{1+2\gamma\sigma+(\gamma\beta)^2}{2} & \gamma \beta\sin{\psi}\\
-\gamma \beta\sin{\psi}& 1+\gamma \sigma-\tfrac{1+2\gamma\sigma+(\gamma\beta)^2}{2}\end{bmatrix}\\
&=\frac{2}{1+2\gamma \sigma+(\gamma \beta)^2}\begin{bmatrix} \tfrac{1}{2}(1-(\gamma\beta)^2) & \gamma \beta\sin{\psi}\\
-\gamma \beta\sin{\psi}& \tfrac{1}{2}(1-(\gamma\beta)^2)\end{bmatrix}
\end{align*}
where we have used
\begin{align*}
(1+\gamma\sigma)^2+(\gamma\beta)^2\sin^2{\psi}=(1+\gamma \beta\cos{\psi})^2+(\gamma \beta)^2\sin^2{\psi}=1+2\gamma\sigma+(\gamma\beta)^2.
\end{align*}
Since $\gamma\beta\sin{\psi}$ is nonnegative, this implies
\begin{align*}
\gamma\beta\sin{\psi}=\sqrt{1-2\gamma\sigma+(\gamma\beta)^2-(1+\gamma\sigma)^2} = \gamma\sqrt{\beta^2-\sigma^2}.
\end{align*}
Therefore, the reflected resolvent is
\begin{align*}
R_{\gamma A} 
&=\frac{2}{1+2\gamma \sigma+(\gamma \beta)^2}\begin{bmatrix} \tfrac{1}{2}(1-(\gamma\beta)^2) & \gamma\sqrt{\beta^2-\sigma^2}\\
-\gamma\sqrt{\beta^2-\sigma^2}& \tfrac{1}{2}(1-(\gamma\beta)^2)\end{bmatrix}.
\end{align*}
Now, let us introduce polar coordinates of the elements:
\begin{align*}
\delta(\cos{\xi},\sin{\xi}) = \left(\tfrac{1}{2}(1-(\gamma\beta)^2),\gamma\sqrt{\beta^2-\sigma^2}\right),
\end{align*}
which gives reflected resolvent
\begin{align}
R_{\gamma A} 
&=\frac{2\delta}{1+2\gamma \sigma+(\gamma \beta)^2}\begin{bmatrix} \cos{\xi} & \sin{\xi}\\
-\sin{\xi}& \cos{\xi}\end{bmatrix}.
\label{eq:A_sm_lip_refl_res}
\end{align}
The angle $\xi$ in the polar coordinate satisfies
\begin{align*}
\tan{\xi} = \tfrac{2\gamma\sqrt{\beta^2-\sigma^2}}{1-(\gamma\beta)^2}
\end{align*}
and since the numerator is nonnegative, $\xi=\arctan_2\left(\tfrac{2\gamma\sqrt{\beta^2-\sigma^2}}{1-(\gamma\beta)^2}\right)$ where $\arctan_2$ is defined in \eqref{eq:atan2}. For the radius $\delta$ in the polar coordinate, we get
\begin{align*}
\delta^2&=\delta^2(\cos^2{\psi}+\sin^2{\psi})\\
&=\tfrac{1}{4}(1-(\gamma\beta)^2)^2+\gamma^2(\beta^2-\sigma^2)\\
&=\tfrac{1}{4}(1-2(\gamma\beta)^2+(\gamma\beta)^4)+\gamma^2(\beta^2-\sigma^2)\\
&=\tfrac{1}{4}(1-2\gamma\sigma+(\gamma\beta)^2)(1+2\gamma\sigma+(\gamma\beta)^2)
\end{align*}
and (since $\delta>0$)
\begin{align}
2\delta&=\sqrt{(1-2\gamma\sigma+(\gamma\beta)^2)(1+2\gamma\sigma+(\gamma\beta)^2)}
\label{eq:delta_A_sm_lip}
\end{align}
It remains to compute the factor in \eqref{eq:A_sm_lip_refl_res}. Using \eqref{eq:delta_A_sm_lip}, we get
\begin{align*}
\frac{2\delta}{1+2\gamma\sigma+(\gamma\beta)^2}&=\frac{\sqrt{(1-2\gamma\sigma+(\gamma\beta)^2)(1+2\gamma\sigma+(\gamma\beta)^2)}}{1+2\gamma\sigma+(\gamma\beta)^2}\\
&=\sqrt{\frac{1-2\gamma\sigma+(\gamma\beta)^2}{1+2\gamma\sigma+(\gamma\beta)^2}}\\
&=\sqrt{1-\frac{4\gamma\sigma}{1+2\gamma\sigma+(\gamma\beta)^2}}
\end{align*}
This completes the proof.

\section{Proofs to results in Section~\ref{sec:A_sm_c}}
\label{app:pf_A_sm_c}

\subsection{Proof to Theorem~\ref{thm:A_sm_c_contraction}}

We know from Lemma~\ref{lem:res_coco} and Definition~\ref{def:cocoercive} that $J_A$ is $(1+\sigma)$-cocoercive, i.e., that it satisfies
\begin{align}
\langle J_Ax-J_Ay,x-y\rangle \geq (1+\sigma)\|J_Ax-J_Ay\|^2
\label{eq:JA_coco}
\end{align}
for all $x,y\in\hilbert$. From Proposition~\ref{prp:res_prop_lipschitz_subdiff}, we know that $J_A$ is $\tfrac{\beta}{2(1+\beta)}$-averaged, i.e., that it satisfies (see \cite[Proposition~4.25(iv)]{bauschkeCVXanal})
\begin{align}
2(1-\tfrac{\beta}{2(1+\beta)})\langle J_Ax-J_Ay,x-y\rangle\geq (1-\tfrac{\beta}{1+\beta})\|x-y\|^2+\|J_Ax-J_Ay\|^2
\label{eq:JA_avg}
\end{align}
for all $x,y\in\hilbert$.
Let $\alpha=\tfrac{\beta}{2(1+\beta)}$ and $\delta=\tfrac{1}{1+\sigma}$ and define the set $\mathcal{R}$ of pairs of points $(x,y)\in\hilbert\times\hilbert$ as:
\begin{align}
\mathcal{R}=\left\{(x,y)~|~\langle J_Ax-J_Ay,x-y\rangle\geq \left(\tfrac{\delta}{2}+\tfrac{(1-\alpha-\delta/2)^2-\alpha^2+(\delta/2)^2}{2(1-\alpha-\delta/2)}\right)\|x-y\|^2\right\}.
\label{eq:R_def}
\end{align}
We also define the closure of set of remaining pairs of points
$\mathcal{R}_c=\overbar{(\hilbert\times\hilbert)\backslash\mathcal{R}}$, i.e.,
\begin{align}
\mathcal{R}_c = \left\{(x,y)~|~\langle J_Ax-J_Ay,x-y\rangle\leq(\tfrac{\delta}{2}+\tfrac{(1-\alpha-\delta/2)^2-\alpha^2+(\delta/2)^2}{2(1-\alpha-\delta/2)})\|x-y\|^2\right\}.
\label{eq:Rc_def}
\end{align}
Obviously, the contraction factor for $R_A$ is the worst-case contraction factor for pairs of points in $\mathcal{R}$ and $\mathcal{R}_c$. 
\subsection*{Contraction factor on $\mathcal{R}$}
First, we provide a contraction factor for pairs of points in $\mathcal{R}$. Since $R_A=2J_A-\id$, we have
\begin{align*}
\|R_Ax-R_Ay\|^2&=4\|J_Ax-J_Ay\|^2-4\langle J_Ax-J_Ay,x-y\rangle+\|x-y\|^2\\
&\leq (4\delta-4)\langle J_Ax-J_Ay,x-y\rangle+\|x-y\|^2\\
&\leq \left(4(\delta-1)\left(\tfrac{\delta}{2}+\tfrac{(1-\alpha-\delta/2)^2-\alpha^2+(\delta/2)^2}{2(1-\alpha-\delta/2)}\right)+1\right)\|x-y\|^2
\end{align*}
where $\delta=\tfrac{1}{1+\sigma}\in(0,1)$ and $\alpha=\tfrac{\beta}{2(1+\beta)}$ and the inequalities follow from \eqref{eq:JA_coco} and the definition of $\mathcal{R}$ in \eqref{eq:R_def}.
\subsection*{Contraction factor on $\mathcal{R}_c$}
Next, we provide a contraction factor for pairs of points in $\mathcal{R}_c$. Since $R_A=2J_A-\id$, we conclude that
\begin{align*}
\|R_Ax&-R_Ay\|^2\\
&=4\|J_Ax-J_Ay\|^2-4\langle J_Ax-J_Ay,x-y\rangle+\|x-y\|^2\\
&\leq (4(2(1-\alpha)-1)\langle J_Ax-J_Ay,x-y\rangle+(1-4(1-2\alpha)))\|x-y\|^2\\
&\leq \left(4(1-2\alpha)\left(\tfrac{\delta}{2}+\tfrac{(1-\alpha-\delta/2)^2-\alpha^2+(\delta/2)^2}{2(1-\alpha-\delta/2)}\right)+1-4+8\alpha\right)\|x-y\|^2,
\end{align*}
where we have used that $\alpha\in(0,\tfrac{1}{2})$, \eqref{eq:JA_avg}, and the definition of $\mathcal{R}_c$ in \eqref{eq:Rc_def}.
\subsection*{Contraction factor of $R_A$}
Here, we show that the contraction factors on $\mathcal{R}$ and $\mathcal{R}_c$ are identical, and we simplify the expression to get a final contraction factor for the reflected resolvent $R_A$. That the contraction factors are identical is shown by verifying that the difference between is zero:
\begin{align*}
&4(1-2\alpha)\left(\tfrac{\delta}{2}+\tfrac{(1-\alpha-\delta/2)^2-\alpha^2+(\delta/2)^2}{2(1-\alpha-\delta/2)}\right)+1-4+8\alpha\\
&\qquad\qquad\qquad\qquad\qquad\qquad\qquad-4(\delta-1)\left(\tfrac{\delta}{2}+\tfrac{(1-\alpha-\delta/2)^2-\alpha^2+(\delta/2)^2}{2(1-\alpha-\delta/2)}\right)-1\\
&=4(2-2\alpha-\delta)\left(\tfrac{\delta}{2}+\tfrac{(1-\alpha-\delta/2)^2-\alpha^2+(\delta/2)^2}{2(1-\alpha-\delta/2)}\right)-4+8\alpha\\
&=2(2-2\alpha-\delta)\delta+4\left((1-\alpha-\delta/2)^2-\alpha^2+(\delta/2)^2\right)-4+8\alpha\\
&=2(2-2\alpha-\delta)\delta+4\left((1-2\alpha-\delta+\alpha\delta+\alpha^2+(\delta/2)^2-\alpha^2+(\delta/2)^2\right)-4+8\alpha\\
&=4\delta-4\alpha\delta-2\delta^2+4-8\alpha-4\delta+4\alpha\delta+2\delta^2-4+8\alpha=0.
\end{align*}
Next, we simplify this contraction factor by inserting $\delta=\tfrac{1}{1+\sigma}$ and $\alpha=\tfrac{\beta}{2(1+\beta)}$. We get
\begin{align*}
&4(\delta-1)\left(\tfrac{\delta}{2}+\tfrac{(1-\alpha-\delta/2)^2-\alpha^2+(\delta/2)^2}{2(1-\alpha-\delta/2)}\right)+1\\
&=4(\tfrac{1}{1+\sigma}-1)\left(\tfrac{1}{2(1+\sigma)}+\tfrac{\left(1-\tfrac{\beta}{2(1+\beta)}-\tfrac{1}{2(1+\sigma)}\right)^2-\left(\tfrac{\beta}{2(1+\beta)}\right)^2+\left(\tfrac{1}{2(1+\sigma)}\right)^2}{2\left(1-\tfrac{\beta}{2(1+\beta)}-\tfrac{1}{2(1+\sigma)}\right)}\right)+1\\
&=4(\tfrac{1}{1+\sigma}-1)\bigg(\tfrac{1}{2(1+\sigma)}\\
&\qquad+\tfrac{1-\tfrac{\beta}{1+\beta}-\tfrac{1}{1+\sigma}+\tfrac{\beta}{2(1+\sigma)(1+\beta)}+\left(\tfrac{\beta}{2(1+\beta)}\right)^2+\left(\tfrac{1}{2(1+\sigma)}\right)^2-\left(\tfrac{\beta}{2(1+\beta)}\right)^2+\left(\tfrac{1}{2(1+\sigma)}\right)^2}{2\left(1-\tfrac{\beta}{2(1+\beta)}-\tfrac{1}{2(1+\sigma)}\right)}\bigg)+1\\
&=4(\tfrac{1}{1+\sigma}-1)\left(\tfrac{1}{2(1+\sigma)}+\tfrac{1-\tfrac{\beta}{1+\beta}-\tfrac{1}{1+\sigma}+\tfrac{\beta}{2(1+\sigma)(1+\beta)}+\tfrac{1}{2(1+\sigma)^2}}{2\left(1-\tfrac{\beta}{2(1+\beta)}-\tfrac{1}{2(1+\sigma)}\right)}\right)+1\\
&=4(\tfrac{1}{1+\sigma}-1)\left(\tfrac{2\left(1-\tfrac{\beta}{2(1+\beta)}-\tfrac{1}{2(1+\sigma)}\right)+2(1+\sigma)\left(1-\tfrac{\beta}{1+\beta}-\tfrac{1}{1+\sigma}+\tfrac{\beta}{2(1+\sigma)(1+\beta)}+\tfrac{1}{2(1+\sigma)^2}\right)}{4(1+\sigma)\left(1-\tfrac{\beta}{2(1+\beta)}-\tfrac{1}{2(1+\sigma)}\right)}\right)+1\\
&=4(\tfrac{1}{1+\sigma}-1)\left(\tfrac{2\left(1-\tfrac{\beta}{2(1+\beta)}-\tfrac{1}{2(1+\sigma)}\right)+2\left((1+\sigma)-\tfrac{\beta(1+\sigma)}{1+\beta}-1+\tfrac{\beta}{2(1+\beta)}+\tfrac{1}{2(1+\sigma)}\right)}{4(1+\sigma)\left(1-\tfrac{\beta}{2(1+\beta)}-\tfrac{1}{2(1+\sigma)}\right)}\right)+1\\
&=4(\tfrac{1}{1+\sigma}-1)\left(\tfrac{2+2\sigma-\tfrac{2\beta(\sigma+1)}{1+\beta}}{4(1+\sigma)\left(1-\tfrac{\beta}{2(1+\beta)}-\tfrac{1}{2(1+\sigma)}\right)}\right)+1\\
&=4(\tfrac{1}{1+\sigma}-1)\left(\tfrac{2\beta+2+2\beta\sigma+2\sigma-2\beta(\sigma+1)}{4(1+\sigma)(1+\beta)\left(1-\tfrac{\beta}{2(1+\beta)}-\tfrac{1}{2(1+\sigma)}\right)}\right)+1\\
&=4(\tfrac{1}{1+\sigma}-1)\left(\tfrac{2+2\sigma}{2\left(2(1+\sigma)(1+\beta)-\beta(1+\sigma)-(1+\beta)\right)}\right)+1\\
&=4(\tfrac{1}{1+\sigma}-1)\left(\tfrac{2+2\sigma}{2(2+2\sigma+2\beta+2\beta\sigma-\beta-\beta\sigma-1-\beta)}\right)+1\\
&=4(\tfrac{1}{1+\sigma}-1)\left(\tfrac{2(1+\sigma)}{2(1+2\sigma+\beta\sigma)}\right)+1\\
&=4(\tfrac{-\sigma}{1+\sigma})\left(\tfrac{2+2\sigma}{2(1+2\sigma+\beta\sigma)}\right)+1\\
&=1-\tfrac{4\sigma}{1+2\sigma+\beta\sigma}.
\end{align*}
Taking the square root concludes the proof.

\subsection{Proof to Proposition~\ref{prp:A_sm_c_refl_contr}}

We start by computing the resolvent and reflected resolvent of $\gamma A$.  The resolvent of $\gamma A$ is given by
\begin{align*}
J_{\gamma A}&=(I+\gamma A)^{-1}\\
&=\begin{bmatrix}\tfrac{\gamma d(1+c\cos{\psi})}{(1+c\cos{\psi})^2+c^2\sin^2{\psi}}+1 &
\tfrac{\gamma dc\sin{\psi}}{(1+c\cos{\psi})^2+c^2\sin^2{\psi}}\\
-\tfrac{\gamma dc\sin{\psi}}{(1+c\cos{\psi})^2+c^2\sin^2{\psi}}&
\tfrac{\gamma d(1+c\cos{\psi})}{(1+c\cos{\psi})^2+c^2\sin^2{\psi}}+1
\end{bmatrix}^{-1}\\
&=\begin{bmatrix}\tfrac{\gamma d(1+c\cos{\psi})}{1+2c\cos{\psi}+c^2}+1 &
\tfrac{\gamma dc\sin{\psi}}{1+2c\cos{\psi}+c^2}\\
-\tfrac{\gamma dc\sin{\psi}}{1+2c\cos{\psi}+c^2}&
\tfrac{\gamma d(1+c\cos{\psi})}{1+2c\cos{\psi}+c^2}+1
\end{bmatrix}^{-1}\\
            &=\begin{bmatrix}\gamma 
\sigma+1 &
\tfrac{\gamma c\sin{\psi}\sigma}{1+c\cos{\psi}}\\
-\tfrac{\gamma dc\sin{\psi}\sigma}{1+c\cos{\psi}}&
\gamma \sigma+1
\end{bmatrix}^{-1}\\
&=\begin{bmatrix}\gamma \sigma+1 &
\gamma c\sin{\psi}\sigma\beta/d\\
-\gamma c\sin{\psi}\sigma\beta/d&
\gamma \sigma+1
\end{bmatrix}^{-1}\\
&=\tfrac{1}{(\gamma \sigma+1)^2+(\gamma c\sigma\beta\sin{\psi}/d)^2}\begin{bmatrix}
\gamma \sigma+1 & -\gamma  c \sigma\beta\sin{\psi}/d\\
\gamma  c \sigma\beta\sin{\psi}/d&\gamma \sigma+1
\end{bmatrix}
\end{align*}
where $\sigma$ and $\beta$ are defined in \eqref{eq:sig_sm_beta_coco}.
The reflected resolvent $R_{\gamma A}$ is given by
\begin{align*}
R_{\gamma A}&=2J_{\gamma A}-I\\
&=\tfrac{2}{(\gamma \sigma+1)^2+(\gamma  c\sigma\beta\sin{\psi}/d)^2}\\
&\quad\times\begin{bmatrix}
\gamma  \sigma+1-\tfrac{(\gamma \sigma+1)^2+(\gamma  c\sigma\beta\sin{\psi}/d)^2}{2} & -\gamma  c \sigma\beta\sin{\psi}/d\\
\gamma c \sigma\beta\sin{\psi}/d&\gamma  \sigma+1-\tfrac{(\gamma \sigma+1)^2+(\gamma  c\sigma\beta\sin{\psi}/d)^2}{2}
\end{bmatrix}.
\end{align*}
To simplify this expression, we note that
\begin{align*}
\beta/\sigma&=\tfrac{d(1+2c\cos{\psi}+c^2)}{d(1+c\cos{\phi})(1+c\cos{\psi})}=\tfrac{(1+2c\cos{\psi}+c^2\pm c^2\cos^2{\psi})}{(1+2c\cos{\phi}+c^2\cos^2{\psi})}=\\
&=1+\tfrac{c^2(1-\cos^2{\psi})}{(1+c\cos{\psi})^2}=1+\tfrac{c^2\sin^2{\psi}}{(1+c\cos{\psi})^2}=1+\tfrac{\beta^2 c^2\sin^2{\psi}}{d^2}.
\end{align*}
This implies that
\begin{align}
 \nonumber (\gamma\sigma+1)^2+(\gamma c\sigma\beta\sin{\psi}/d)^2&=1+2\gamma\sigma+(\gamma\sigma)^2(1+c^2\beta^2\sin^2{\psi}/d^2)\\
&=1+2\gamma\sigma+\sigma\beta\gamma^2.
\label{eq:refl_den}
\end{align}
Using this equality, we can simplify the reflected resolvent expression:
\begin{align*}
R_{\gamma A}&=\tfrac{2}{1+2\gamma\sigma+\sigma\beta\gamma^2}\begin{bmatrix}
\tfrac{1}{2}(1-\beta\sigma\gamma^2) & -\gamma\sqrt{\sigma(\beta-\sigma)}\\
\gamma\sqrt{\sigma(\beta-\sigma)}&\tfrac{1}{2}(1-\beta\sigma\gamma^2)
\end{bmatrix},
\end{align*}
since $\gamma c\sigma\beta\sin{\psi}/d>0$. Now, let us write the matrix elements using polar coordinates:
\begin{align*}
\delta(\cos{\xi},\sin{\xi})=\left(\tfrac{1}{2}(1-\sigma\beta\gamma^2),\gamma\sqrt{\sigma(\beta-\sigma)}\right).
\end{align*}
This gives the reflected resolvent:
\begin{align}
R_{\gamma A}=\tfrac{2\delta}{1+2\gamma\sigma+\sigma\beta\gamma^2}\begin{bmatrix}
\cos{\xi} & -\sin{\xi}\\
\sin{\xi} & \cos{\xi}
\end{bmatrix}.
\label{eq:refl_gA}
\end{align}
The angle $\xi$ in the polar coordinates satisfies
\begin{align*}
\tan{\xi} = \tfrac{2\gamma\sqrt{\sigma(\beta-\sigma)}}{1-\sigma\beta\gamma^2}.
\end{align*}
The numerator is always nonnegative, so $\xi$ is given by $\xi=\arctan_2\left(\tfrac{2\gamma\sqrt{\sigma(\beta-\sigma)}}{1-\sigma\beta\gamma^2}\right)$ with $\arctan_2$ defined in \eqref{eq:atan2}. The radius $\delta$ in the polar coordinates satisfies
\begin{align*}
\delta^2&=\delta^2(\cos^2{\xi}+\sin^2{\xi})\\
&=(\tfrac{1-\sigma\beta\gamma^2}{2})^2+\gamma^2\sigma(\beta-\sigma)\\
&=\tfrac{1}{4}-\tfrac{\sigma\beta\gamma^2}{2}+\tfrac{(\sigma\beta\gamma^2)^2}{4}+\sigma\beta\gamma^2-(\gamma\sigma)^2\\
&=\tfrac{1}{4}+\tfrac{\sigma\beta\gamma^2}{2}+\tfrac{(\sigma\beta\gamma^2)^2}{4}-(\gamma\sigma)^2\\
&=\tfrac{1}{4}(1-2\gamma\sigma+\gamma^2\sigma\beta)(1+2\gamma\sigma+\gamma^2\sigma\beta)
\end{align*}
and (since $\delta>0$)
\begin{align}
2\delta = \sqrt{(1-2\gamma\sigma+\gamma^2\sigma\beta)(1+2\gamma\sigma+\gamma^2\sigma\beta)}.
\label{eq:delta_coco}
\end{align}

It remains to compute the factor in \eqref{eq:refl_gA}. Using \eqref{eq:delta_coco}, we conclude
\begin{align*}
\frac{2\delta}{1+2\gamma\sigma+\sigma\beta\gamma^2}&=\frac{\sqrt{(1-2\gamma\sigma+\gamma^2\sigma\beta)(1+2\gamma\sigma+\gamma^2\sigma\beta)}}{1+2\gamma\sigma+\gamma^2\sigma\beta}\\
&=\sqrt{\frac{1-2\gamma\sigma+\gamma^2\sigma\beta}{1+2\gamma\sigma+\gamma^2\sigma\beta}}\\
&=\sqrt{1-\frac{4\gamma\sigma}{1+2\gamma\sigma+\gamma^2\sigma\beta}}.
\end{align*}
This completes the proof.

\end{appendix}


\bibliographystyle{plain}
\bibliography{references}

\end{document}